\documentclass[11pt]{article}

\newcommand{\C}{\mathbb{C}}
\newcommand{\CB}{\mathcal{CB}_n}
\newcommand{\meas}{\mathcal{M}}
\newcommand{\ketpsi}{\ket{\psi}}
\newcommand{\keti}{\ket{i}}
\newcommand{\U}{U}
\newcommand{\Exp}{\textup{Exp}}
\newcommand{\CX}{CX}

\newcommand{\ketzeron}{\ket{0}^{\otimes n}}
\newcommand{\universal}{PU} 
\newcommand{\F}{\mathcal{F}}
\newcommand{\Fq}{\mathcal{F}_\text{quant}}
\newcommand{\G}{\mathcal{G}}
\newcommand{\A}{\mathcal{A}}
\newcommand{\qddots}{
  \raisebox{0.3em}{\ensuremath{\ddots}}%
}
\newcommand{\qldots}{
  \raisebox{0.3em}{\ensuremath{\ldots}}%
}

\usepackage{amsmath,amsthm,amssymb, thm-restate}
\usepackage{todonotes}
\usepackage{physics}
\usepackage{qcircuit}
\usepackage{float}
\usepackage[left=2.7cm,right=2.7cm,top=2.7cm,bottom=2.7cm]{geometry}
\usepackage[boxed]{algorithm2e} 
\usepackage{stmaryrd}
\usepackage{systeme}
\usepackage{hyperref}
\hypersetup{
colorlinks=true,
breaklinks=true,
urlcolor= blue,
linkcolor= black, 
citecolor=black} 
\usepackage{graphicx} 
\usepackage{authblk}
\usepackage[toc,page]{appendix}
\usepackage[pagewise]{lineno}

\newtheorem{theorem}{Theorem}
\newtheorem{definition}[theorem]{Definition}
\newtheorem{remark}[theorem]{Remark}
\newtheorem{example}[theorem]{Example}
\newtheorem{proposition}[theorem]{Proposition}

\title{An introduction to variational quantum algorithms for combinatorial optimization problems}
\author[1,2]{Camille Grange}
\author[1]{Michael Poss}   
\author[1]{Eric Bourreau}

\affil[1]{LIRMM, University of Montpellier, CNRS, France}
\affil[2]{SNCF, Technology, Innovation and Group Projects Department, France}

\date{}

\begin{document}

\setlength{\parskip}{0mm}
\setlength{\abovedisplayskip}{3mm}
\setlength{\belowdisplayskip}{3mm}
\setlength{\parindent}{4mm}
\allowdisplaybreaks

\maketitle
\setcounter{page}{1}
\renewcommand{\thepage}{\arabic{page}}

\begin{abstract}
    Noisy intermediate-scale quantum computers (NISQ computers) are now readily available, motivating many researchers to experiment with Variational Quantum Algorithms (VQAs). Among them, the Quantum Approximate Optimization Algorithm (QAOA) is one of the most popular one studied by the combinatorial optimization community. In this tutorial, we provide a mathematical description of the class of Variational Quantum Algorithms, assuming no previous knowledge of quantum physics from the readers. We introduce precisely the key aspects of these hybrid algorithms on the quantum side (parametrized quantum circuit) and the classical side (guiding function, optimizer). We devote a particular attention to QAOA, detailing the quantum circuits involved in that algorithm, as well as the properties satisfied by its possible guiding functions. Finally, we discuss the recent literature on QAOA, highlighting several research trends. \\
    
    \noindent\textbf{keywords:} Variational Quantum Algorithm, QAOA, Combinatorial Optimization, Metaheuristics
\end{abstract}

\section{Introduction}
This tutorial aims at providing the operational research community with some background knowledge of quantum computing and explore the particular branch of heuristic quantum algorithms for combinatorial optimization. Today, the quantum information theory community is excited by the emergence of quantum computers first theorized in the 1980s. Indeed, since then, some theoretical advantages of quantum algorithms compared with classical algorithms have been proved for several problems. For instance, the Quantum Phase Estimation~\cite{QPE}, which is the subroutine of Shor's algorithm~\cite{Shor}, enables the latter to solve the integer factorization problem efficiently. Moreover, Grover Search~\cite{Grover} finds an element in an unstructured base with a quadratic speedup compared with classical algorithms. This algorithm is also used as a subroutine in dynamic programming algorithms to improve their exponential complexity for problems such as the Travelling Salesman Problem and the Minimum Set Cover~\cite{ProgDyn}. More recently, a polynomial speedup that relies on a Quantum Interior Point Method~\cite{Kerenidis_LP} applies to Linear Programming, Semi-Definite Programming, and Second-Order Cone Programming~\cite{Kerenidis_SOCP}.
Furthermore, a quantum subroutine for the Simplex also provides a polynomial speedup~\cite{Simplex}. 

These algorithms prove quantum advantages on numerous problems, but their implementations require a lot of quantum resources with high quality. Specifically, they need quantum computers with many qubits that can interact two by two and quantum operations that can be applied in a row on qubits without generating noise, namely they can support deep-depth quantum circuits. Such quantum computers are believed to be built in the medium term but current devices are noisy intermediate-scale quantum computers, usually referred to as NISQ computers~\cite{NISQ}. Therefore, NISQ computers cannot handle the implementations of such algorithms. These physical limitations encouraged quantum algorithm theory to look into \emph{lighter} quantum algorithms to implement them on quantum computers today. It consists of hybrid algorithms that require both quantum and classical resources. Indeed, the classical part overcomes the limited and noisy quantum resources, whereas the quantum part still takes partial advantage of quantum information theory. This tutorial focuses on a particular type of hybrid algorithm, that is the class of Variational Quantum Algorithms (VQAs).

Variational Quantum Algorithms are heuristic algorithms that alternate between a quantum circuit and a classical optimizer. They tackle optimization problems of the form
\begin{equation}
\label{eq:min_f}
\min_{x\in\{0,1\}^n} f(x)\,,
\end{equation}
where $f$ is any function defined on $\{0,1\}^n$. 
VQAs are of great interest to the quantum information theory community today because they have the convenient property of an adjustable quantum circuits' depth, making them implementable on the current NISQ computers.
The variational approach of VQAs~\cite{Cerezo} consists of probing the initial search space $\{0,1\}^n$ with a relatively small set of parameters optimized classically. Specifically, these parameters describe a probability distribution over the search space. This description results from a sequence of quantum logic gates: this is a quantum circuit. VQAs take advantage of the quantum computing principle~\cite{Nielsen} that prepares a probability distribution over an exponential search space in a short sequence of quantum gates. The choice of the quantum circuit, and in particular the number of parameters, is a huge stake for VQAs~\cite{Nannicini_Hyb_Perf}.
On the one hand, the more parameters, the more precise the search space probing. On the other hand, too many parameters can harden the classical optimization part. The function that drives this optimization, which depends on the probability distribution and $f$, is also a non-trivial choice and worth investigating~\cite{CVaR}. 

In this tutorial, we provide a mathematical description of Variational Quantum Algorithms and focus on one of them, specifically the Quantum Approximate Optimization Algorithm (QAOA)~\cite{QAOAFarhi}. We shall also devote particular attention to problems in which $f$ is a polynomial function. 
In Section~\ref{sect:basics}, we give a brief introduction to the basics of quantum computing for combinatorial optimization. Notice that the model of quantum computation considered is the gate-based model, also called the circuit model. Then, in Section~\ref{sect:VQA}, we describe the general class of VQAs,
starting by defining the different parts that constitute and define a VQA, namely the quantum circuit (also called \emph{ansatz} in the literature), the classical optimizer, and the guiding function driving the classical optimization. Then, we characterize each part with properties that should be valid for potential theoretical guarantees. Eventually, in Section~\ref{sect:QAOA}, we focus on a particular case of VQAs, the Quantum Approximate Optimization Algorithm. We describe the necessary reformulation of the initial problem~\eqref{eq:min_f} into a Hermitian matrix (also called Hamiltonian in the literature) to be solved by QAOA and analyze this algorithm in light of the previous properties of Section~\ref{sect:VQA}. We also provide a universal decomposition of the QAOA quantum circuit for the general case where $f$ is polynomial while, as far as we know, only the quadratic $f$ case had been treated in the literature so far. Finally, we give a condensed overview of empirical trends and theoretical limitations of QAOA.
We remind some helpful basic notions of linear algebra in  Appendix~\ref{sec:appendix}. Furthermore, the proofs of technical properties of the function optimized in VQAs and specific constructions of the quantum circuit in QAOA are deferred to Appendices~\ref{sec:proof:guiding_func} and~\ref{sec:proof:quantum_circuit}, respectively.

\section{Basics of quantum computing for combinatorial optimization}
\label{sect:basics}
This section aims at providing the basic notions of quantum computing necessary for the understanding of the quantum resolution of combinatorial problems. 
\subsection{Quantum bits}
Let $\ket{0}$ and $\ket{1}$ denote the basic states of our quantum computer (the counterpart of states 0 and 1 in classical computers). The first building block of quantum algorithms is the quantum bit, also called qubit.

\begin{definition}[Qubit]
\label{def:qubit}
We define a qubit as 
\begin{equation}
\label{eq:qubit}
\ket{q}=q_0\ket{0}+q_1\ket{1}\,, 
\end{equation}
where $(q_0,q_1)\in \C^2$ is a pair of complex numbers that satisfies the normalizing condition
$$
|q_0|^2 +|q_1|^2 = 1\,.
$$
We say that $q_0$ and $q_1$ are the \emph{coordinates} of $\ket{q}$ in the basis $(\ket{0}, \ket{1})$.  
\end{definition}
It is often convenient to use the matrix representation of $\ket{0}$ and $\ket{1}$, namely
$$
\ket{0} =  \begin{pmatrix}
    1 \\
    0
\end{pmatrix}
\mbox{ and }
\ket{1} =  \begin{pmatrix}
    0 \\
    1
\end{pmatrix}\,.
$$
With this matrix representation, the qubit $\ket{q}$ defined in~\eqref{eq:qubit} is equal to 
$$ \ket{q} = \begin{pmatrix}
q_0\\
q_1
\end{pmatrix}\,.$$
\begin{example}
\label{ex:+ and -}
Important examples of one-qubit states are $\frac{\ket{0} + \ket{1}}{\sqrt{2}}$ and $\frac{\ket{0} - \ket{1}}{\sqrt{2}}$. These are usually denoted as $\ket{+}$ and $\ket{-}$, respectively. 
\end{example}
The algorithms studied in this paper typically manipulate quantum states of larger dimension. Specifically, let $n$ denote the number of qubits that our quantum computing device is able to manipulate simultaneously. An $n$-qubit state is defined by $2^n$ complex numbers that satisfy the normalizing condition and represent the normal decomposition in the canonical basis.
\begin{definition}[Canonical basis]
The canonical basis of an $n$-qubit state is the set 
\begin{equation*}
    \CB = \left(\bigotimes_{k=1}^n\ket{i^{(k)}}, (i^{(1)},\ldots,i^{(n)}) \in \{0,1\}^n\right)\,,
\end{equation*}
where $i^{(k)}$ represents the state of the $k$-th qubit and $\otimes$ is the tensor product defined in Appendix~\ref{sec:appendix}. 
The canonical basis is the set of all possible combinations of tensor products of $n$ one-qubit basis states, $\ket{0}$ and $\ket{1}$. Thus, in the matrix representation, each canonical basis state is a column vector of $2^n$ components with one component equal to 1 and the rest equal to 0. It directly results that the size of the canonical basis of an $n$-qubit state is $2^n$. For more readability, we omit to write the tensor product between qubits, i.e., we refer to the canonical basis as 
\begin{equation*}
    \CB = (\ket{i}, i \in \{0,1\}^n)\,.
\end{equation*}
\end{definition}

The matrix representation of canonical basis states results from the definition above. We illustrate it for the canonical basis of two qubits:
\begin{equation*}
    \ket{00} = \begin{pmatrix}
1\\
0
\end{pmatrix}\otimes \begin{pmatrix}
1\\
0
\end{pmatrix} = \begin{pmatrix}
1\\
0\\
0\\
0
\end{pmatrix}\,, ~~ ~~
    \ket{01} = \begin{pmatrix}
1\\
0
\end{pmatrix}\otimes \begin{pmatrix}
0\\
1
\end{pmatrix} = \begin{pmatrix}
0\\
1\\
0\\
0
\end{pmatrix}\,,
\end{equation*}
\begin{equation*}
    \ket{10} = \begin{pmatrix}
0\\
1
\end{pmatrix}\otimes \begin{pmatrix}
1\\
0
\end{pmatrix} = \begin{pmatrix}
0\\
0\\
1\\
0
\end{pmatrix}\,,~~ ~~
    \ket{11} = \begin{pmatrix}
0\\
1
\end{pmatrix}\otimes \begin{pmatrix}
0\\
1
\end{pmatrix} = \begin{pmatrix}
0\\
0\\
0\\
1
\end{pmatrix}\,.
\end{equation*}

\begin{definition}[$n$-qubit state]
An $n$-qubit state $\ketpsi$ is a normalized linear combination of the basis states in $\CB$,
\begin{equation*}
    \ketpsi = \mathlarger\sum_{i \in \{0,1\}^n} \psi_i\ket{i}\,,
\end{equation*}
where $(\psi_i)_{i\in\{0,1\}^n} \in \C^{2^n}$ are its coordinates, which satisfy the normalizing condition
\begin{equation}
\label{eq:normalizingCondition}
     \braket{\psi}=\sum_{i \in \{0,1\}^n}|\psi_i|^2 = 1\,.
\end{equation}

\end{definition}
For $n = 1$, we find the definition of a qubit (Definition \ref{def:qubit}), as expected.

\begin{example}
\label{psi+}
For instance, $\frac{\ket{00}+\ket{11}}{\sqrt{2}}$, called $\ket{\Phi^+}$, is a two-qubit state. Its matrix representation is:
\begin{align*}
    \ket{\Phi^+} &= \frac{\ket{00}+\ket{11}}{\sqrt{2}} \\
    &= \frac{1}{\sqrt{2}}\left(\begin{pmatrix}
        1\\0
    \end{pmatrix}\otimes \begin{pmatrix}
        1\\0
    \end{pmatrix} + \begin{pmatrix}
        0\\1
    \end{pmatrix}\otimes \begin{pmatrix}
        0\\1
    \end{pmatrix}\right)\\
    &= \frac{1}{\sqrt{2}}\left(\begin{pmatrix}
        1\\0\\0\\0
    \end{pmatrix} + \begin{pmatrix}
        0\\0\\0\\1
    \end{pmatrix}\right)\\
    &= \frac{1}{\sqrt{2}}\begin{pmatrix}
        1\\0\\0\\1
    \end{pmatrix}.
\end{align*}
\end{example}
\begin{remark}
Our notation $(\psi_i)_{i\in\{0,1\}^n} \in \C^{2^n}$ for the coordinates of $\ketpsi$ in basis $\CB$ is used to simplify the presentation throughout. It is, however, uncommon in the quantum computing literature where different bases may be used.
\end{remark}

\subsection{Quantum gates}

In the model of gate-based quantum computation, qubits are manipulated with quantum gates. Mathematically speaking, these quantum gates are modeled by unitary matrices. More precisely, a quantum gate that manipulates $n$-qubit states is a matrix in $\mathcal{M}_{2^n}(\mathbb{C})$ that modifies the $2^n$ complex coefficients of a quantum state such that they still satisfy the normalizing condition \eqref{eq:normalizingCondition}. 

\begin{definition}[Unitary matrix]
A matrix $\U
\in \mathcal{M}_{2^n}(\mathbb{C})$ is unitary if its inverse is equal to its conjugate transpose (denoted by the dagger symbol $\dagger$) : 
$$\U\U^{\dagger} = \U^{\dagger}\U= I\,.$$ Notice that the application of a quantum gate is logically reversible.
\end{definition}
We easily see that a unitary matrix $\U$ preserves the normalizing condition since, denoting $\ket{\psi'}=\U\ketpsi$, we have
$$
\braket{\psi'}=\bra{\psi}\U^{\dagger}\U\ketpsi=\braket{\psi}\,.
$$
As an illustration, we describe below all the possible one-qubit gates, i.e., all the unitary matrices in $\mathcal{M}_{2}(\mathbb{C})$.
\begin{example}[Generic one-qubit gate]
A generic one-qubit gate $\U= \begin{pmatrix}
    a & b \\
    c & d
\end{pmatrix} \in \mathcal{M}_{2}(\mathbb{C})$ satisfies
\begin{equation*}
    \begin{cases}
    |a|^2+|b|^2 = 1 \\
    a\bar{c} + b\bar{d}= 0\\
    \bar{a}c + \bar{b}d = 0\\
    |c|^2+|d|^2 = 1
    \end{cases}
\end{equation*}
Its application on a qubit $\ket{q} = q_0\ket{0} + q_1\ket{1}$ is:
\begin{equation}
    \label{eq:Ugate}
    \U\ket{q} =  \begin{pmatrix}
    a & b \\
    c & d
\end{pmatrix} \begin{pmatrix}
    q_0 \\
    q_1
\end{pmatrix} = \begin{pmatrix}
    aq_0 + bq_1 \\
    cq_0 + dq_1
\end{pmatrix} = (aq_0 + bq_1)\ket{0} + (cq_0 + dq_1)\ket{1} = \ket{q'}
\end{equation}
\end{example}
We introduce next the circuit representation of a quantum gate, which is a useful formalism to represent unitary matrices. Specifically, Figure~\ref{fig:Ucircuit} represents the application of the quantum gate $U$ to the one-qubit state $\ket{q}$, as in~\eqref{eq:Ugate}, by:
\begin{figure}[H]
$$
    \Qcircuit @C=1em @R=.7em {
\lstick{\ket{q}} & \gate{U} & \rstick{\ket{q'}} \qw \\
}
$$
\caption{Circuit of gate $U$ applying to $\ket{q}$.}
\label{fig:Ucircuit}
\end{figure}
This representation easily generalizes to $n$-qubit quantum gates, where each \emph{horizontal line}
is associated with a qubit. 

\subsection{Quantum circuits}
\label{sec:q_circuit}
We can manipulate quantum gates to build other quantum gates in two different ways. The first one is the \emph{composition}, and the other one is the \emph{tensor product}. 

\begin{definition}[Composition of quantum gates]
The composition of gates only operates between gates acting on the same qubits. Let $k$ be the number of involved qubits. The composition of $U_1\in \mathcal{M}_{2^k}(\C)$ and $ U_2 \in \mathcal{M}_{2^k}(\C)$ consists of the application of $U_1$ followed by the application of $U_2$.
The matrix representation of this composition is the product $U_2U_1$. 
The circuit representation of this composition is illustrated on Figure~\ref{fig:composition} for $k = 3$. 
\begin{figure}[H]
$$
    \Qcircuit @C=1em @R=.7em {
& \multigate{2}{U_1} & \multigate{2}{U_2} & \qw \\
& \ghost{U_1} & \ghost{U_2} & \qw \\
& \ghost{U_1} & \ghost{U_2} & \qw
}
$$
\caption{Composition of $U_1$ and $U_2$.}
\label{fig:composition}
\end{figure}
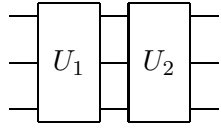
It can be seen as a series sequence of gates. 
\end{definition}
Notice that we read from right to left in the matrix representation, and from left to right in the circuit representation. Moreover, in the circuit representation, qubits are numbered in ascending order from top to bottom which is important for the tensor product circuit representation that follows. 

\begin{definition}[Tensor product of quantum gates]
The tensor product of gates only operates between gates acting on different qubits. Suppose $U_1\in\mathcal{M}_{2^k}(\C)$ applies on the first $k$ qubits and $U_2\in\mathcal{M}_{2^{k'}}(\C)$ applies on the $k'$ following ones. Their tensor product is the application of $U_1$ and $U_2$ on respective qubits in parallel. 
The matrix representation of this tensor product is $U_1\otimes U_2$ (Definition~\ref{def:tensorProduct}). 
The circuit representation is depicted on Figure~\ref{fig:tensorProduct} for $k = 3$ and $k'=2$. 
\begin{figure}[H]
$$
    \Qcircuit @C=1em @R=.7em {
& \multigate{2}{U_1} & \qw  \\
& \ghost{U_1} & \qw \\
& \ghost{U_1} & \qw \\
& \multigate{1}{U_2} & \qw  \\
& \ghost{U_2} & \qw
}    
$$
\caption{Tensor product of $U_1$ and $U_2$.}
\label{fig:tensorProduct}
\end{figure}
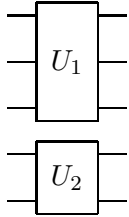
\end{definition}
Notice that when we apply a quantum gate on $k$ qubits of an $n$-qubit system, it supposes that we apply identity gate $I = \begin{pmatrix} 1 & 0 \\ 0 & 1\end{pmatrix}$ on the qubits not concerned. For instance, let us consider a $3$-qubit system on which we apply $U\in\mathcal{M}_{2}(\C)$ on qubit number $2$. The matrix representation of the resulting 3-qubit gate is $I\otimes U\otimes I$, and its circuit representation is illustrated on Figure~\ref{fig:identityTProduct}, where the application of $I$ is usually replaced by a simple wire.
\begin{figure}[H]
$$
    \Qcircuit @C=1em @R=.7em {
& \qw & \qw  \\
& \gate{U}  & \qw \\
&  \qw  & \qw \\
}  $$ 
\caption{Application of $U$ to qubit number 2.}
\label{fig:identityTProduct}
\end{figure}
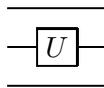
One readily verifies that both composition and tensor product transform unitary matrices into a resulting unitary matrix. 

Throughout, we consider that a quantum algorithm is a quantum circuit acting on $n$ qubits, that is, a sequence of quantum gates' compositions and/or tensor products. These quantum gates can be $k$-qubit gates, for $k\in[n]$.  
However, quantum gates involving many qubits are typically not implementable natively on quantum computers and need to be decomposed into smaller and simpler gates. This set of small gates can be considered as the quantum counterpart of the elementary logic gates used in classical circuit computing to assess the circuit complexity of a classical algorithm.
Thus, an $n$-qubit quantum algorithm is described by a unitary matrix in $\mathcal{M}_{2^n}(\C)$, and we decompose it as a sequence of universal gates (Definition~\ref{def:universalGates}) to obtain the complexity of the quantum algorithm.
\begin{definition}[Set of universal gates]
\label{def:universalGates}
A set of quantum gates $\universal$ is \emph{universal} if we can decompose any $n$-qubit quantum gate through a circuit composed solely of the gates in \universal.
\end{definition}

Fortunately, there exist different universal sets of quantum gates. We introduce below such a set formed of four types of gates. First, we consider the three following families of one-qubit gates, each of which is parametrized by a real number $\theta \in \mathbb{R}$:
\begin{align}
\label{def:RX}
    R_X(\theta) &=  \begin{pmatrix}
    \cos{\frac{\theta}{2}} & -i\sin{\frac{\theta}{2}} \\
    -i\sin{\frac{\theta}{2}} & \cos{\frac{\theta}{2}}
\end{pmatrix}\,,\\
\label{def:RY}
R_Y(\theta) &=  \begin{pmatrix}
    \cos{\frac{\theta}{2}} & -\sin{\frac{\theta}{2}} \\
    \sin{\frac{\theta}{2}} & \cos{\frac{\theta}{2}}
\end{pmatrix}\,,\\
\label{def:RZ}
    R_Z(\theta) &=  \begin{pmatrix}
    e^{-i\frac{\theta}{2}} & 0 \\
    0 & e^{i\frac{\theta}{2}}
\end{pmatrix}\,.
\end{align}
Notice that these gates are often referred to as rotation gates because they correspond to rotations in a certain representation of qubits, known as the Bloch sphere~\cite{BlochSphere}. Second, we consider the two-qubit gate $\CX$: 
\begin{equation}
\label{def:CX}
    \CX = \begin{pmatrix}
    1 & 0 & 0 & 0\\
    0 & 1 & 0 & 0\\
    0 & 0 & 0 & 1\\
    0 & 0 & 1 & 0
\end{pmatrix}\,.
\end{equation}
$\CX$ applies gate $X = \begin{pmatrix}
    0 & 1 \\
    1 & 0
\end{pmatrix}$ on the second qubit if and only if the first qubit is in state $\ket{1}$. Indeed, for instance, we have 
\begin{align*}
    \CX \ket{01} = \begin{pmatrix}
    1 & 0 & 0 & 0\\
    0 & 1 & 0 & 0\\
    0 & 0 & 0 & 1\\
    0 & 0 & 1 & 0
\end{pmatrix}\begin{pmatrix}
    0 \\
    1 \\
    0 \\
    0 
\end{pmatrix}
= \begin{pmatrix}
    0 \\
    1 \\
    0 \\
    0 
\end{pmatrix} = \ket{01}\,,
\end{align*}
and 
\begin{align*}
    \CX \ket{11} = \begin{pmatrix}
    1 & 0 & 0 & 0\\
    0 & 1 & 0 & 0\\
    0 & 0 & 0 & 1\\
    0 & 0 & 1 & 0
\end{pmatrix}\begin{pmatrix}
    0 \\
    0 \\
    0 \\
    1 
\end{pmatrix}
= \begin{pmatrix}
    0 \\
    0 \\
    1 \\
    0 
\end{pmatrix} = \ket{10}\,.
\end{align*}
In other words, we can define $\CX$ on the canonical basis as $$\CX\ket{a,b} = \ket{a,b\oplus a}\,,$$ for $a,b\in\{0,1\}$, where $\oplus$ is the addition modulo 2.

\begin{remark}[Notation on gates indexes]
\label{remark:notationGates}
Henceforth, we use the notation $R_{X,i}$ for the application of $R_X$ on qubit $i$ (and the application of identity matrix on the remaining qubits). We do the same with $R_{Y,i}$ and $R_{Z,i}$. We note $\CX_{i,j}$ the application of $\CX$ gate to qubits $i$ and $j$ : $X$ is applied to qubit $j$ if and only if qubit $i$ is in state $\ket{1}$.
\end{remark}
\begin{theorem}[Universal gates \cite{Nielsen}]
\label{th:universality}
The set of one-qubit gates and the  $\CX$ gate~\eqref{def:CX} is universal. Thus, because any one-qubit gate is the composition of rotation gates~\eqref{def:RX}--\eqref{def:RZ}, the set $\universal = \{R_X(\alpha), R_Y(\beta), R_Z(\gamma), \CX : \alpha, \beta, \gamma \in \mathbb{R}\}$ is universal. 
\end{theorem}
In comparison, the sets of gates \{NAND\}, \{NOR\}, \{NOT, AND\} and \{NOT, OR\} are universal for classical computation. Indeed, we can compute any arbitrary classical function with them.
In view of the above, we typically consider that the quantum counterpart of the classical number of elementary operations is the number of universal gates used to decompose the circuit. Of course, this decomposition depends on the set of universal gates $\universal$ considered, but the number of gates required is the same with any set, modulo a  multiplicative constant~\cite{Nielsen}. 
Thereafter, we only consider the set $\universal$ defined in Theorem~\ref{th:universality}. This choice is motivated by the algorithms we study in this paper, as it will be convenient to express them with this set. 
Furthermore, if we consider the family of circuits $(\mathcal{C}_n)_{n\in\mathbb{N}}$ where $\mathcal{C}_n$ is a circuit on $n$ qubits and is decomposed on $\mathcal{O}(poly(n))$ universal quantum gates, then this family is said to be \emph{efficient}.

\subsection{Non-classical behaviors}
The quantum algorithms we present in this paper rely on three characteristics of quantum states with no classical equivalent: measurement, superposition, and entanglement. 
Let us present these notions through the bare minimum mathematical background.

\subsubsection{Measurement}

We need to measure a quantum state $\ketpsi$ to get information from it. Otherwise, no information is accessible. The peculiar property of measurement is that it only extracts partial information from the quantum state: the single measurement output of $\ketpsi$ is a bitstring. 

\begin{definition}[Measurement]
In the gate-based quantum model, the measurement $\meas$ of an $n$-qubit state $\ketpsi = \mathlarger\sum_{i \in \{0,1\}^n} \psi_i\ket{i}$ outputs the $n$-bitstring $i$ with probability $|\psi_i|^2$. After having been measured, state $\ketpsi$ no longer exists: it has been replaced by the state $\ket{i}$. 
\end{definition}
For example, measuring qubit $\ket{q} = q_0\ket{0} + q_1\ket{1}$ outputs $0$ with probability $|q_0|^2$ and $1$ with probability $|q_1|^2$, and changes the state $\ket{q}$ to $\ket{0}$ and $\ket{1}$, respectively. A measurement appears as a loss of information. Indeed, we describe an $n$-qubit state by $2^n$ normalized complex coefficients, but we only extract an $n$-bitstring after measuring it. The perfect knowledge of the probabilities representing a given quantum state $\ketpsi$, namely the square module of each of its coordinates $(|\psi_i|^2)_{i\in\{0,1\}^n} \in [0,1]^{2^n}$, can be obtained only if we measure $\ketpsi$ an infinite number of times. Notice that it requires resetting state $\ketpsi$ after each measurement. 
\begin{remark}[Sampling of quantum states]
\label{rem:sampling}
In reality we are limited to approximating a given quantum state through sampling. In particular, if $\ketpsi$ is the result of an algorithm, this means we have to repeat the same algorithm for every measurement of $\ketpsi$ we wish to perform. 
\end{remark}

\subsubsection{Superposition}

Classically, the state of an $n$-bit computer is given by a bitstring in $\{0,1\}^n$. We have seen so far that the state of an $n$-qubit quantum computer is given by its coordinates $(\psi_i)_{i\in\{0,1\}^n} \in \C^{2^n}$, which satisfy $\mathlarger\sum_{i\in\{0,1\}^n}|\psi_i|^2=1$. In general, more than one of the coordinates are different from 0, meaning that measuring $\ketpsi$ may result in different bitstring $i \in \{0,1\}^n$. 
\begin{definition}[Superposition]
A quantum state $\ketpsi$ is said in superposition if $ \ketpsi = \mathlarger\sum_{i \in \{0,1\}^n} \psi_i\ket{i}$ where there are at least two terms with non-zero coefficients in the sum. A quantum state that is not a basis state is in superposition.
\end{definition}
The following Hadamard gate is the usual one-qubit gate that produces superposition starting from a canonical basis state. 
\begin{example}[Hadamard gate]
\label{ex:HadamardGate}
The Hadamard gate 
$H = \frac{1}{\sqrt{2}}\begin{pmatrix}
    1 & 1 \\
    1 & -1
\end{pmatrix}$ 
is essential in quantum computing because it creates superposition starting from a basis state. We obtain the state $\ket{+}$, respectively  $\ket{-}$, of Example~\ref{ex:+ and -} by applying $H$ on $\ket{0}$, respectively $\ket{1}$:
\begin{align*}
H\ket{0} &= \frac{\ket{0} + \ket{1}}{\sqrt{2}} = \ket{+}\,,\\
H\ket{1} &= \frac{\ket{0} - \ket{1}}{\sqrt{2}} = \ket{-}\,.
\end{align*}
\end{example}
\begin{example}[Uniform superposition]
\label{ex:unif_superposition}
The two states $\ket{+}$ and $\ket{-}$ are in uniform superposition because both have the same probability of being measured as $0$ or $1$. In general, an $n$-qubit state uniformly superposed is equal to $\frac{1}{\sqrt{2^n}}\mathlarger\sum_{j \in \{0,1\}^n} e^{i\alpha_j}\ket{j}$, with $\alpha_j \in [0, 2\pi[$, $\forall j \in \{0,1\}^n$. In what follows, we shall often use the uniformly superposed $n$-qubit state $\ket{+}^{\otimes n} = \frac{1}{\sqrt{2^n}}\mathlarger\sum_{i \in \{0,1\}^n}\ket{i}$.
\end{example}
Notice that applying an $n$-qubit quantum gate $\U$ to $\ketpsi$ possibly modifies the $2^n$ coordinates of $\ketpsi$ since
$$
\U\ketpsi = \sum_{i\in\{0,1\}^n}\psi'_i\keti\,,
$$
where possibly each $\psi'_i$ differs from $\psi_i$. This is the case, for instance, when applying the tensor product of $n$ Hadamard gates, each applied to a qubit initially in state $\ket{0}$, specifically
\begin{equation}
\label{eq:Hnket0}
H^{\otimes n}\ket{0}^{\otimes n} = \bigotimes_{i=1}^n H\ket{0} = \bigotimes_{i=1}^n \left(\frac{\ket{0}+\ket{1}}{\sqrt{2}} \right) = \frac{1}{\sqrt{2^n}}\mathlarger\sum_{i \in \{0,1\}^n}\ket{i} = \ket{+}^{\otimes n}\,.
\end{equation}
Equation~\eqref{eq:Hnket0} illustrates the potential benefit of quantum circuits: applying $\mathcal{O}(n)$ universal one-qubit gates impacts the exponentially many coefficients of $\ketpsi$. Indeed, one readily verifies that $H = R_X(\pi)R_Y(\frac{\pi}{2})$ modulo a global phase\footnote{Two quantum states $\ketpsi$ and $\ket{\psi '} = e^{i\alpha}\ketpsi$, with $\alpha\in[0,2\pi[$, that only differ by a global phase are indiscernible by measurement. Thus, we do not consider global phase of quantum states nor quantum gates.}, so $H^{\otimes n}$ amounts to applying $2n$ universal one-qubit gates. 

\subsubsection{Entanglement}
\label{sec:entanglement}
Each quantum state is either a product state or an entangled state. Entanglement has the peculiar and helpful property that we can apply a circuit only on a part of the $n$-qubit system, and as a result, the whole system is affected.
\begin{definition}[Product state]
An $n$-qubit state is a product state if it is the tensor product of $n$ one-qubit states.
In other words, an $n$-qubit state $\ketpsi$ is a product state if it exists $2n$ complex coefficients $(q_0^{(j)}, q_1^{(j)})_{j \in [n]}$ such that 
\begin{equation*}
    \ketpsi = \bigotimes_{j = 1}^{n}(q_0^{(j)}\ket{0} + q_1^{(j)}\ket{1}),  \text{ with }  |q_0^{(j)}|^2+|q_1^{(j)}|^2 = 1, \forall j \in [n]\,.
\end{equation*}
Thus, each state of a qubit that composes $\ketpsi$ can be described independently of the states of the other.
\end{definition}
If an $n$-qubit state is not a product state, it is an entangled state.
\begin{definition}[Entangled state]
An $n$-qubit state $\ketpsi = \mathlarger\sum_{i \in \{0,1\}^n} \psi_i\ket{i}$ is entangled if the numerical values of its coordinates $(\psi_i)_{i\in\{0,1\}^n} \in \C^{2^n}$ admit no solution $(q_0^{(j)}, q_1^{(j)})_{j \in [n]} \in \mathbb{C}^{2n}$ to the system
\begin{equation*}
\begin{cases}
   \mathlarger\sum_{i \in \{0,1\}^n} \psi_i\ket{i} = \mathlarger\bigotimes_{j = 1}^{n}(q_0^{(j)}\ket{0} + q_1^{(j)}\ket{1})\\
   |q_0^{(j)}|^2+|q_1^{(j)}|^2 = 1 , \forall j \in \llbracket 1, n \rrbracket
\end{cases}
\end{equation*}
It means that operations performed on some coordinates of the entangled state can affect the other coordinates without direct operations on them.
\end{definition}
Notice that when an $n$-qubit system is entangled, it makes no sense to speak about qubit number $k \in [n]$ because isolated qubits are not defined. Notice also that there is a difference between superposition and entanglement. A state is in superposition if at least two non-zero coefficients are in its basis decomposition. A state is entangled if it cannot be written as a tensor product of independent qubits, implying that it is in superposition. 

To illustrate the notion of entanglement, 
we consider the following quantum circuit on Figure~\ref{fig:entanglingCircuit}. Starting from a two-qubit product state $\ket{00}$, it results an entangled state. 
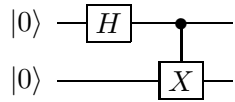
\begin{figure}[H]
$$
    \Qcircuit @C=1em @R=.7em {
\lstick{\ket{0}} & \gate{H} & \ctrl{1} &  \qw \\
\lstick{\ket{0}} & \qw & \gate{X}  & \qw
}
$$
\caption{Entangling circuit.}
\label{fig:entanglingCircuit}
\end{figure}
This circuit is the application of Hadamard gate on qubit $1$, followed by $CX_{1,2}$.
The quantum state that results is $\ket{\Phi^+}$ of Example~\ref{psi+}:
\begin{equation*} 
	\begin{split}
	  CX_{1,2}(H\otimes I)\ket{00} & = CX_{1,2}\left(\frac{1}{\sqrt{2}}(\ket{00} + \ket{10})\right)\\
	 & = \frac{1}{\sqrt{2}}\left(CX_{1,2}\ket{00} + CX_{1,2}\ket{10}\right)\\
	 & = \frac{1}{\sqrt{2}}(\ket{00} + \ket{11})\\
	 &= \ket{\Phi^+}\,.
	\end{split}
\end{equation*}
We prove now by contradiction that $\ket{\Phi^+}$ is entangled.
Suppose that $\ket{\Phi^+}$ is a product state. Thus, there exists $(q_0^{(0)},q_1^{(0)})\in\mathbb{C}^2$ such as $|q_0^{(0)}|^2 + |q_1^{(0)}|^2 = 1$ and $(q_0^{(1)},q_1^{(1)})\in\mathbb{C}^2$ such as $|q_0^{(1)}|^2 + |q_1^{(1)}|^2 = 1$, satisfying:
\begin{align*}
    \ket{\Phi^+} &= (q_0^{(0)}\ket{0} + q_1^{(0)}\ket{1})\otimes (q_0^{(1)}\ket{0} + q_1^{(1)}\ket{1})\\
    &= q_0^{(0)}q_0^{(1)}\ket{00} +q_0^{(0)}q_1^{(1)}\ket{01} + q_1^{(0)}q_0^{(1)}\ket{10} +q_1^{(0)}q_1^{(1)}\ket{11}\,.
\end{align*}
Because the decomposition of a quantum state's coordinates is unique in the canonical basis, it follows by identification: 
\begin{equation*}
    \begin{cases}
   q_0^{(0)}q_0^{(1)} &= \frac{1}{\sqrt{2}} \\
   q_0^{(0)}q_1^{(1)} &= 0 \\
   q_1^{(0)}q_0^{(1)} &= 0 \\
   q_1^{(0)}q_1^{(1)} &= \frac{1}{\sqrt{2}}
\end{cases}
\end{equation*}
This equation system admits no solution. We deduce that $\ket{\Phi^+} = \frac{\ket{00} + \ket{11}}{\sqrt{2}}$ is not a product state, hence it is entangled. 

\section{Variational Quantum Algorithms for optimization}
\label{sect:VQA}
Some of the quantum algorithms that tackle combinatorial optimization problems are Variational Quantum Algorithms (VQAs)~\cite{Cerezo}. They are hybrid algorithms because they require both quantum and classical computations. VQAs are studied today because they represent an alternative approach that reduces the quality and quantity of the quantum resources needed~\cite{Peruzzo}. Specifically, they are designed to run on the NISQ era \cite{NISQ} where quantum computers are noisy with few qubits: for instance, they harness low-depth quantum circuits.

In this section, we consider optimization problems of the form
\begin{equation}
    \underset{x\in\{0,1\}^n}{\min}f(x)\,, \tag{\ref{eq:min_f}}
\end{equation}
where $f$ is any function defined on $\{0,1\}^n$. We note $\F$ the set of optimal solutions. 

\subsection{General description}
\label{sect:VQA_gen_desc}
Variational Quantum Algorithms (VQAs) are hybrid algorithms that, given an input $\ketzeron$, alternate between a quantum and a classical part. Henceforth, we note $\ket{0_n}$ the state $\ketzeron$ to ease the reading.
Let us first provide a high-level description
of the key elements of VQAs that are detailed in this section. Let $d\in\mathbb{N}$, the three key elements of VQAs are:
\begin{itemize}
    \item a \emph{parametrized quantum circuit}, $U : \mathbb{R}^d \rightarrow \mathcal{M}_{2^n}(\mathbb{C})\,,$
    \item a \emph{guiding function}, $g : \mathbb{R}^d \rightarrow \mathbb{R}\,,$
    \item and a \emph{classical optimizer}, which is an algorithm $\A$ that optimizes $g$ over space $\mathbb{R}^d\,.$
\end{itemize}
These elements, the parametrized quantum circuit and the two others, are detailed and illustrated in Subsections~\ref{sec:qu_part} and~\ref{sec:class_part}, respectively. Before that,
we provide a general overview of VQAs in Algorithm~\ref{alg:VQA}.

The main idea of a VQA is as follows. The main loop of the algorithm executes the following steps until the classical optimizer stops according to a given stopping criterion. First, given $\theta\in\mathbb{R}^d$, the \emph{quantum part} executes the parametrized quantum circuit $U(\theta)$ for several times to sample the quantum state (see Remark~\ref{rem:sampling}). It results in a distribution probability over $\{0,1\}^n$ that we note $\{p_\theta(x) : x\in\{0,1\}^n\}$. Second, the \emph{classical part} computes the cost of this state through the evaluation of the guiding function according to the sampling results, specifically, $g(\theta) = G(p_\theta, f)$ for a given function $G$ that is constructed from $f$ and $p_\theta$. More precisely, for a given $\theta$, the computation of $G$ requires only the values $f(x)$ for $x$ such that $p_\theta(x) > 0$. Eventually, this cost value is given to the classical optimizer $\A$, which outputs a new parameter in order to minimize $g$. Notice that, between the two parts the best-found solution is possibly updated by a classical computer.

\begin{algorithm}[!ht]
\label{alg:VQA}
 \KwIn{guiding function $g$, parametrized quantum circuit $U$, classical optimizer $\A$, initial parameter $\theta_0$}
\caption{Variational Quantum Algorithm}
 \KwOut{approximate minimum $f^*$ and the corresponding minimizer $x^*$}
 $\theta \leftarrow \theta_0$\;
 $f^* \leftarrow +\infty$\;
 $x^* \leftarrow 0^n$\;
 \While{stopping criterion of the classical optimizer is false}
{
\Begin(\textbf{quantum part}){
 \For{size of the sampling}{
 measure the
 state $U(\theta)\ketzeron$, outputting $x$ for some $x\in \{0,1\}^n$\;}
let $p_\theta(x)$ be the frequency of $x\in \{0,1\}^n$ in the above sampling\;}
 \Begin(\textbf{solution update}){
 \For{$x$ such that $p_\theta(x) > 0$}
{ compute $f(x)$\;
 \If{$f(x) < f^*$}{
 $f^* \leftarrow f(x)$\;
 $x^* \leftarrow x$\;}}}
 \Begin(\textbf{classical part}){
 compute $g(\theta)=G(p_\theta, f)$\;
 given $g(\theta)$ (and possibly its derivatives), $\A$ outputs $\theta'$\; 
 $\theta \leftarrow \theta'$\;
}}
\end{algorithm}

In the following subsections, we detail each part of VQAs and show how specific choices of $g$ and $U$ ensure that VQAs optimize $f$. Let us begin with the quantum part. 

\subsection{Quantum part}
\label{sec:qu_part}
The quantum part of VQAs applies a quantum circuit on the $n$-qubit system that constitutes the quantum computer. Importantly, \emph{variational} in VQAs stands for the parametrization of the quantum circuit. Let $d \in \mathbb{N}$ be the number of parameters. 
\begin{definition}
A parametrized quantum circuit is a continuous function $U:\mathbb{R}^d\rightarrow \mathcal{M}_{2^n}(\mathbb{C})$ mapping any $\theta \in \mathbb{R}^d$ to unitary matrix $\U(\theta)$.
\end{definition}
As defined in Subsection~\ref{sec:q_circuit}, a quantum circuit is a sequence of universal quantum gates' compositions and/or tensor products. Thus, all coefficients of matrix $U(\theta)$ are continuous functions on $\mathbb{R}^d$.

\begin{example}
A simple example of a parametrized quantum circuit for $n=3$ and $d=3$ is depicted in Figure~\ref{fig:ex_circuit}.
\begin{figure}[H]
$$
    \Qcircuit @C=1em @R=.7em {
\lstick{\ket{0}} & \gate{R_Y(\theta_1)} & \qw  \\
\lstick{\ket{0}} & \gate{R_Y(\theta_2)}  & \qw \\
\lstick{\ket{0}} & \gate{R_Y(\theta_3)}  & \qw \\
}    
$$
\caption{Quantum circuit parametrized by $(\theta_1, \theta_2, \theta_3) \in \mathbb{R}^3$.}
\label{fig:ex_circuit}
\end{figure}
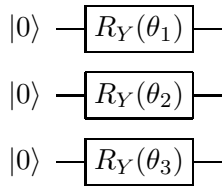
The expression of $U(\theta)$ for this circuit is as follows: $\forall \theta = (\theta_1, \theta_2, \theta_3) \in \mathbb{R}^3$,
\begin{align*}
    U(\theta) &= R_Y(\theta_1)\otimes R_Y(\theta_2)\otimes R_Y(\theta_3)\\
    &= 
    \begin{pmatrix}
    \cos{\frac{\theta_1}{2}} & -\sin{\frac{\theta_1}{2}} \\
    \sin{\frac{\theta_1}{2}} & \cos{\frac{\theta_1}{2}}
\end{pmatrix} \otimes \begin{pmatrix}
    \cos{\frac{\theta_2}{2}} & -\sin{\frac{\theta_2}{2}} \\
    \sin{\frac{\theta_2}{2}} & \cos{\frac{\theta_2}{2}}
\end{pmatrix} \otimes 
\begin{pmatrix}
    \cos{\frac{\theta_3}{2}} & -\sin{\frac{\theta_3}{2}} \\
    \sin{\frac{\theta_3}{2}} & \cos{\frac{\theta_3}{2}}
\end{pmatrix}\\
&= \begin{pmatrix}
    c_1 c_2 c_3 & -c_1 c_2 s_3 & -c_1 s_2 c_3 & c_1 s_2 s_3 & -s_1 c_2 c_3 & s_1 c_2 s_3 & s_1 s_2 c_3 & -s_1 s_2 s_3 \\
    c_1 c_2 s_3 & c_1 c_2 c_3 & -c_1 s_2 s_3 & -c_1 s_2 c_3 & -s_1 c_2 s_3 & -s_1 c_2 c_3 & s_1 s_2 s_3 & s_1 s_2 c_3\\
    c_1 s_2 c_3 & -c_1 s_2 s_3 & c_1 c_2 c_3 & -c_1 c_2 s_3 & -s_1 s_2 c_3 & s_1 s_2 s_3 & -s_1 c_2 c_3 & s_1 c_2 s_3\\
    c_1 s_2 s_3 & c_1 s_2 c_3 & c_1 c_2 s_3 & c_1 c_2 c_3 & -s_1 s_2 s_3 & -s_1 s_2 c_3 & -s_1 c_2 s_3 & -s_1 c_2 c_3 \\
    s_1 c_2 c_3 & -s_1 c_2 s_3 & -s_1 s_2 c_3 & s_1 s_2 s_3 & c_1 c_2 c_3 & -c_1 c_2 s_3 & -c_1 s_2 c_3 & c_1 s_2 s_3 \\
    s_1 c_2 s_3 & s_1 c_2 c_3 & -s_1 s_2 s_3 & -s_1 s_2 c_3 & c_1 c_2 s_3 & c_1 c_2 c_3 & -c_1 s_2 s_3 & -c_1 s_2 c_3\\
    s_1 s_2 c_3 & -s_1 s_2 s_3 & s_1 c_2 c_3 & -s_1 c_2 s_3 & c_1 s_2 c_3 & -c_1 s_2 s_3 & c_1 c_2 c_3 & -c_1 c_2 s_3\\
    s_1 s_2 s_3 & s_1 s_2 c_3 & s_1 c_2 s_3 & s_1 c_2 c_3 & c_1 s_2 s_3 & c_1 s_2 c_3 & c_1 c_2 s_3 & c_1 c_2 c_3
\end{pmatrix}\,,
\end{align*}
where $c_i = \cos{\frac{\theta_i}{2}}$ and $s_i = \sin{\frac{\theta_i}{2}}$ for $i\in[3]$.
\end{example}
\begin{remark}
\label{rem:cont_relax}
    The use of the generalized circuit to $n$ qubits $\mathlarger\bigotimes_{i=1}^n{R_{Y,i}(\theta_i)}$ amounts to a continuous relaxation of the $\{0,1\}$-problem~\eqref{eq:min_f}, where each decision variable $x_i \in[0,1]$ is represented by a rotation angle $\theta_i\in [0,2\pi]$ as follows:
    \begin{equation*}
        x_i = \left(\cos{\frac{\theta_i}{2}}\right)^2\,.
    \end{equation*}
\end{remark}

\subsection{Classical part}
\label{sec:class_part}
The classical part of VQAs consists of a classical optimization over the parameters $\theta \in \mathbb{R}^d$. The classical optimizer essentially aims at finding the optimal parameters $\theta^*$ that lead to optimal solutions of the initial problem~\eqref{eq:min_f} with high probability, specifically, such that 
\begin{equation}
\label{eq:VQAoptimal}
    \sum_{s\in \F}|\bra{s}U(\theta^*)\ket{0_n}|^2 \geq 1-\epsilon\,,
\end{equation}
for small $\epsilon>0$. 
Henceforth, we use notation $\ket{x}$ instead of $\ket{i}$ to efficiently recall that we deal with solutions of optimization problems.

The classical part is characterized by two aspects: the function that guides the optimization and the optimizer itself.

\subsubsection{Guiding function}
\label{subsec:guid_function}
Let $g : \mathbb{R}^d \to \mathbb{R}$ be the guiding function, formally defined in Definition~\ref{def:guiding_function} below, that the classical optimizer minimizes. The function $g$ acts as a link between the quantum and classical parts. For a given $\theta \in \mathbb{R}^d$, we evaluate $U(\theta)\ket{0_n}$ \emph{according to $f$} as we will exemplify below. 
Notice that $f$ and $g$ are distinct since $g$ is defined on $\mathbb{R}^d$ and outputs a quality measure of an $n$-qubit quantum state whereas $f$ is defined on $\{0,1\}^n$. 
Let us denote $\Fq = \left\{\mathlarger\sum_{s\in\F} \psi_{s}\ket{s} : \mathlarger\sum_{s\in\F} |\psi_{s}|^2 = 1\right\}$ as the set of quantum states that are superpositions of optimal solutions of problem~\eqref{eq:min_f}.
Naturally, we would like to define $g$ such that minimizing $g$ tends to minimize $f$.
\begin{definition}[Guiding function]
\label{def:guiding_function}
Let $g: \mathbb{R}^d \to \mathbb{R}$ be a function and $\G$ be its set of minimizers. We call $g$ a guiding function for $f$ with respect to $U$ if $g$ is continuous and 
\begin{equation}
\label{eq:incl_guiding}
    \{U(\theta)\ket{0_n} : \theta \in \G\} \subseteq \Fq\,.
\end{equation}
\end{definition}
In other words, optima of a guiding function $g$ must lead to optima of $f$ or superpositions of optima of $f$. Indeed, Equation~\eqref{eq:incl_guiding} implies that measuring the quantum state $U(\theta^*)\ketzeron$ for $\theta^*\in \G$ outputs with probability 1 an optimal solution of the initial problem $s\in\F$.
Thus, minimizing $g$ amounts to minimizing $f$, and finding $\F$ is done by finding $\G$. The latter is found with the classical optimizer.
Without any information on $\F$, we need to choose a quantum circuit $U$ such that any optimal solution $s\in\F$ is reachable, specifically,
\begin{equation}
    \label{eq:incl_U}
    \{U(\theta)\ket{0_n} : \theta\in\mathbb{R}^d\} \supseteq \CB\,.
\end{equation}
Notice that this condition is weak and easily satisfied. For instance, the circuit depicted in Figure~\ref{fig:ex_circuit} satisfies this condition.
If ever one is interested in finding all optimal solutions of problem~\eqref{eq:min_f}, the circuit and the guiding function should satisfy instead the stronger condition
\begin{equation}
    \label{eq:optim_incl}
    \{U(\theta)\ket{0_n} : \theta \in \G\} = \Fq\,.
\end{equation}
In that case, $U$ satisfying~\eqref{eq:incl_U} is not enough. Without any information on $\F$, we need to choose $U$ that can reach any $n$-qubit quantum states, specifically, 
\begin{equation*}
    \{U(\theta)\ket{0_n} : \theta\in\mathbb{R}^d\} = \left\{\mathlarger\sum_{x\in\{0,1\}^n} \psi_{x}\ket{x} : \mathlarger\sum_{x\in\{0,1\}^n} |\psi_{x}|^2 = 1\right\}\,.
\end{equation*}

A popular choice for the guiding function in the literature is the mean function
\begin{equation}
\label{def:g_mean}
    g_{\text{mean}}(\theta) = \sum_{x\in\{0,1\}^n} p_\theta(x)f(x)\,,
\end{equation}
where $p_\theta(x) = |\bra{x}U(\theta)\ket{0_n}|^2$ is the probability of finding $x$ when $U(\theta)\ket{0_n}$ is measured. We show next that $g_{\text{mean}}$ is indeed a guiding function according to Definition~\ref{def:guiding_function}, see Appendix~\ref{proof:g_mean_guinding} for a proof.

\begin{restatable}{proposition}{gmeanguinding}\label{prop:g_mean_guinding}
Function $g_{\text{mean}}$ is a guiding function. 
\end{restatable}

\begin{example}
We illustrate the mean function on the 3-qubit quantum circuit $U(\theta)$ depicted in Figure~\ref{fig:ex_circuit}. The generalization of its computation for $n$ qubits is trivial since it needs to replace 3 by $n$.
The single application of rotation gate $R_Y$~\eqref{def:RY} of angle $\theta_i$ on a qubit initially on state $\ket{0}$ is 
\begin{align*}
    R_Y(\theta_i)\ket{0} &= \cos\frac{\theta_i}{2}\ket{0} + \sin\frac{\theta_i}{2}\ket{1}\\
    &= \sum_{j\in \{0,1\}} \cos \frac{\theta_i - j\pi}{2}\ket{j}\,,
\end{align*}
since $\sin(\phi) = \cos(\phi - \frac{\pi}{2})$ for $\phi \in \mathbb{R}$. 
Eventually, the quantum state resulting from $U(\theta)$ is 
\begin{align*}
    U(\theta)\ket{0_3} &= \bigotimes_{i=1}^3 R_{Y,i}(\theta_i)\ket{0}\\
    &= \sum_{j_1,j_2,j_3\in\{0,1\}}\left( \prod_{i=1}^3 \cos\frac{\theta_i - j_i\pi}{2}\right)\ket{j_1j_2j_3}\,.
\end{align*}
Thus, the probability to measure $x = (x_1,x_2,x_3) \in\{0,1\}^3$ is 
\begin{equation}
\label{eq:ex_prob}
    p_\theta(x) = \left(\prod_{i=1}^3 \cos\frac{\theta_i - x_i\pi}{2}\right)^2\,,
\end{equation}
and the expression of $g_{\text{mean}}$ of equation~\eqref{def:g_mean} directly results from it. 
\end{example}

Other functions are compatible with Definition~\ref{def:guiding_function}. One of these functions encountered in the literature is the Gibbs function~\cite{Gibbs} which stems from statistical mechanics. Let $\eta>0$ be a parameter to be set. The Gibbs function is defined as
\begin{equation}
    \label{def:g_gibbs}
    g_{G,\eta}(\theta) = -\ln\left({\sum_{x\in\{0,1\}^n}p_\theta(x)e^{-\eta f(x)}} \right)\,.
\end{equation}
The choice of this function is motivated by the exponential shape that highly rewards the increase of probabilities of low-cost states.
Notice that for small $\eta$, minimizing the Gibbs function is essentially equivalent to minimizing the mean function in the sense that the Taylor series of $g_{G,\eta}$ at first order in $\eta = 0$  gives $g_{G,\eta} = \eta g_\text{mean}$. We show next that $g_{G,\eta}$ is indeed a guiding function according to Definition~\ref{def:guiding_function}, see Appendix~\ref{proof:g_gibbs_guiding} for a proof.

\begin{restatable}{proposition}{ggibbsguiding}\label{prop:g_gibbs_guiding}
Let $\eta>0$. Function $g_{G,\eta}$ is a guiding function. 
\end{restatable}

One might suggest other guiding functions, such as the minimum function 
\begin{equation}
\label{def:g_min}
    g_{\text{min}}(\theta)=\underset{x \hspace{0.5mm}:\hspace{0.5mm} p_\theta(x) > 0}{\min} f(x)\,.
\end{equation}
However, $g_{\text{min}}$ does not verify~\eqref{eq:incl_guiding}, and is not even continuous, hardening its optimization and excluding its choice for the guiding function. 
\begin{example}
\label{ex:g_min}
Let us illustrate that $g_{\text{min}}$ is not a guiding function. 
For that, we consider the circuit of Figure~\ref{fig:ex_circuit} and the following function $f: \{0,1\}^3 \mapsto \mathbb{R}$ to minimize, 
\begin{equation*}
    \begin{cases}
   f(0,0,0) &= 1 \\
   f(x) &= 0\,,~~\forall x \neq (0,0,0)
\end{cases}
\end{equation*}
where $f^* = 0$ is the optimal value.
Function $g_{\text{min}}$ reaches its optimal value $g_{\text{min}}^* = 0$ on the set of its optimizers $$\G =  \mathbb{R}^3\setminus\{(2k\pi, 2k\pi, 2k\pi) : k\in\mathbb{Z}\}\,.$$ However, $$\forall \theta \in \G \setminus\{((2k+1)\pi, (2k+1)\pi, (2k+1)\pi) : k\in\mathbb{Z}\}\,,U(\theta)\ket{0_n} \notin \Fq\,,$$ because there is a non-zero probability of sampling $(0,0,0)$. Thus, $g_{\text{min}}$ violates~\eqref{eq:incl_guiding}.
Moreover, $g_{\text{min}}$ is not continuous. Indeed, $g_{\text{min}}(0,0,0) = 1$, whereas $\forall \epsilon >0\,, g_{\text{min}}(\epsilon,0,0) = 0$. 
\end{example}
The above observation motivated~\cite{CVaR} to suggest another function, the CVaR (Conditional Value-at-Risk) function. 
The CVaR function is the average on the lower $\alpha$-tail of values of $f$ encountered, where $\alpha\in]0,1]$ is a parameter to be set. 
Let $(x_1,\ldots,x_{2^n})$ be the $n$-bitstrings sorted in non decreasing order, namely $f(x_i) \leq f(x_{i+1})$ for any $i\in[2^n - 1]$. Let $N_\alpha$ be the index that delimits the $\alpha$-tail elements of the distribution, specifically,
\begin{equation*}
    N_\alpha = \min \left\{N \geq 1 : \sum_{i=1}^{N} p_\theta(x_i) \geq \alpha\right\}\,.
\end{equation*}
Then, the CVaR function is
\begin{equation}
\label{def:g_cvar}
    g_{C,\alpha}(\theta) =  \frac{1}{\sum_{i=1}^{N_\alpha}p_\theta(x_i)}\sum_{i=1}^{N_\alpha} p_\theta(x_i)f(x_i)\,.
\end{equation}
The special case $\alpha = 1$ implies $g_\alpha = g_{\text{mean}}$, whereas when $\alpha$ approaches zero, we find $g_{\text{min}}$.

The CVaR function is an alternative to the non-smooth minimum function. 
While CVaR does not verify~\eqref{eq:incl_guiding} either, it keeps continuity and still focuses on the best solutions that appear on the probability distribution. 
\begin{example}
Let us illustrate the violation of~\eqref{eq:incl_guiding} by the \textup{CVaR} function, for any $\alpha\in]0,1[$. 
For that, we consider function $f$ of Example~\ref{ex:g_min} with the same quantum circuit of Figure~\ref{fig:ex_circuit}.
Let $\alpha\in ]0,1[$. Let us find $\theta\in\G$ such that $U(\theta)\ket{0_n}\notin\Fq$. In other words, we search $\theta\in\mathbb{R}^3$ such that $$g_{C,\alpha}(\theta) = 0\,\text{ and }\,p_\theta(0,0,0) > 0\,.$$
Let us look at $\theta = (\pi - \epsilon, 0, 0)$, where $\epsilon \in ]0,\pi[$.
According to~\eqref{eq:ex_prob}, we have $$p_\theta(0,0,0) = \cos^2\left(\frac{\pi - \epsilon}{2}\right)>0\,.$$
It remains to choose $\epsilon$ to ensure $g_{C,\alpha}(\theta) = 0$. Namely, we want $\epsilon$ such that 
\begin{equation*}
    \sum_{x\neq (0,0,0)}p_\theta(x) \geq \alpha\,,
\end{equation*}
meaning
\begin{equation*}
    1 - \cos^2\left(\frac{\pi - \epsilon}{2}\right) \geq \alpha\,.
\end{equation*}
This holds for any $\epsilon \leq \pi - 2\arccos{(\sqrt{1-\alpha})}$.
\end{example}
Even if CVaR does not verify~\eqref{eq:incl_guiding}, it can be seen as a pseudo-guiding function, defined just below. Notice that in practice, using the CVaR function seems appropriate because it accepts a probability of measuring optimal solutions after optimizing to be lower than one, such as expressed in~\eqref{eq:VQAoptimal}.

\begin{definition}[Pseudo-guiding function]
\label{def:weak_gu_func}
Let $g: \mathbb{R}^d \to \mathbb{R}$ be a function and $\G$ be its set of minimizers. We call $g$ a pseudo-guiding function for $f$ with respect to $U$ if $g$ is continuous and if there exists $\alpha\in]0,1[$ such that optima of $g$ can lead to non-optimal solutions of $f$ with a probability strictly lower than $1-\alpha$. 
Specifically, let $\theta\in\G$. Thus, either $U(\theta)\ket{0_n}\in\Fq$ or $\mathlarger\sum_{x\notin \F}|\bra{x}U(\theta)\ket{0_n}|^2 < 1-\alpha$.
\end{definition}

We show next that $g_{C,\alpha}$ is indeed a pseudo-guiding function according to Definition~\ref{def:weak_gu_func}, see Appendix~\ref{proof:g_CVaR} for a proof.

\begin{restatable}{proposition}{gCVaR}\label{prop:g_CVaR}
Let $\alpha \in ]0,1[$. Function $g_{C,\alpha}$ is a pseudo-guiding function.
\end{restatable}

\subsubsection{Classical optimizer}
The role of the classical optimizer is to minimize the guiding function. The function $g$ is continuous and is usually differentiable but not convex. Any unconstrained optimization algorithm can be used to minimize $g$, such as local search, descent gradient method, or any black-box optimization algorithm.

To speak in terms of stochastic optimization,
the classical optimizer aims at solving the stochastic programming model under endogenous uncertainty
\begin{equation}
\label{def:min_g}
    \min_\theta \{g(\theta) = \mathbb{E}[ G(\theta, \xi_\theta)]\}\,,
\end{equation}
where the definition of $G$ depends on the choice of a specific guiding function, and $\xi_\theta$ is an endogenous vector that depends on $\theta$. Specifically, $\xi_\theta$ is a discrete random variable, with the set of possible outcomes $\{0,1\}^n$ and the following distribution probability:
\begin{equation*}
    \mathbb{P}(\xi_\theta = x) = p_\theta(x),~~\forall x\in \{0,1\}^n\,.
\end{equation*}

Thus, this problem falls into the class of stochastic dependent-decision probabilities problems~\cite{Optim_stocha_dep}. In practice, the classical optimizer approximates the function by a Monte Carlo estimation as follows:
\begin{equation*}
    \hat{g}_N(\theta) = \frac{1}{N}\sum_{j=1}^{N} G(\theta, \xi_{\theta}^j)\,,
\end{equation*}
where $\{\xi^j_\theta\}_{j\in[N]}$ is a sample of size $N$ from the distribution of $\xi_\theta$. Notice that for a given $\theta$, the quantity $\hat{g}_N(\theta)$ itself is a random variable since its value depends on the sample that has been generated, which is random. In contrast, the value of $g(\theta)$ is deterministic.
 In practice, the classical optimizer iterates the loop that consists of, given a sampling distribution of size $N$ of the quantum state $U(\theta)\ket{0_n}$ (Remark~\ref{rem:sampling}), outputs a $\theta'$. This $\theta'$ is then transmitted to the quantum part.  Eventually, the aim is to output $\theta^*$ such that $U(\theta)\ket{0_n} \subseteq \Fq$. 
In each iteration, the value of $\hat{g}_N(\theta)$ is computed. Notice that according to the Law of Large Numbers~\cite{Shapiro}, $\hat{g}_N(\theta)$ converges with probability one to $g(\theta)$ as $N\rightarrow \infty$.

For instance, for the case of $g=g_\text{mean}$, we have
\begin{equation*}
    G(\theta, \xi_\theta) = f(\xi_\theta)\,.
\end{equation*}
Hence, for each $\xi_\theta^j\in\{0,1\}^n$ sampled, either we compute classically $f(\xi_\theta^j)$ and store it if not already computed, or we get its value. Thus, we can compute $\hat{g}_N(\theta)$. Notice that for the CVaR function, we compute the empirical mean only on the best $\lceil \alpha N \rceil$ values $f(\xi_\theta)$ found by sampling.

The solution returned by the VQA is the minimum value of $f$ and the associated minimizer $x$ encountered while the algorithm runs. 

\section{Quantum Approximate Optimization Algorithm}
\label{sect:QAOA}
We assume throughout this section that the function $f$ to be minimized is polynomial.
First, we reformulate problem~\eqref{eq:min_f} to a more suitable form for quantum optimization. This reformulation is motivated by the quantum adiabatic evolution~\cite{AdiabEvol} that, for a given Hermitian matrix\footnote{A complex square matrix is Hermitian if it is equal to its conjugate transpose.}, approximates the eigenvector with the lowest eigenvalue under certain conditions. For that, we interpret the objective function of problem~\eqref{eq:min_f} as an Hermitian matrix $H_f$ such that each eigenvector $\ket{u_x}$ is matching a classical solution $x\in\{0,1\}^n$ with an eigenvalue equal to $f(x)$, specifically,
\begin{equation*}
    H_f\ket{u_x} = f(x)\ket{u_x}\,.
\end{equation*}
Thus, the solutions of problem~\eqref{eq:min_f} are the solutions corresponding to the lowest eigenvalues of $H_f$. The Quantum Approximate Optimization Algorithm (QAOA)~\cite{QAOAFarhi} presented in this section aims at finding the lowest eigenvalue of $H_f$.

\subsection{Problem reformulation}
The construction of $H_f$ is as follows. First, we transform the $\{0,1\}$-problem~\eqref{eq:min_f} into a $\{-1,1\}$-problem. For that, we apply the following linear transformation: for $x=(x_1,\ldots,x_n)\in\{0,1\}^n$, we define $z=(z_1,\ldots,z_n)\in\{-1,1\}^n$ where 

\begin{equation}
\label{eq:transfo}
    z_i = 1-2x_i\,, \forall i \in [n]\,.
\end{equation}
This leads to the problem
\begin{equation*}
    \underset{z \in \{-1,1\}^n}{\min} f_{\pm}(z)\,,
\end{equation*}
where, for $z\in\{-1,1\}^n$,
\begin{equation*}
    f_{\pm}(z) = \mathlarger\sum_{\alpha = (\alpha_1,...\alpha_n) \in \{0,1\}^n} h_{\alpha}\prod_{i=1}^n z_i^{\alpha_i}\,,
\end{equation*}
where $h_{\alpha} \in \mathbb{R}, \forall \alpha \in \{0,1\}^n$.
Second, we define $H_f$ as
\begin{equation*}
    H_f = \mathlarger\sum_{\alpha = (\alpha_1,...\alpha_n) \in \{0,1\}^n} h_{\alpha}\bigotimes_{i=1}^n Z_i^{\alpha_i}\,,
\end{equation*}
where $Z = \begin{pmatrix}1&0\\0&-1\end{pmatrix}$, $Z^0 = I$ and $Z^1 = Z$. We note $Z_i$ the application of $Z$ to qubit $i$. Notice that $Z$ is equal to the universal gate $R_{Z}(\pi)$ modulo a global phase (see~\eqref{def:RZ}). This construction of $H_f$ leads to the following property. 
\begin{proposition}
The eigenvectors of $H_f$ are the canonical basis $\ket{x}\in\CB$ with eigenvalues that are the cost of the solutions $f(x)$, specifically, 
\begin{equation}
\label{eq:eigenvectors}
    \forall \ket{x}\in\CB,\hspace{2mm} H_f\ket{x} = f(x)\ket{x}\,.
\end{equation}
\end{proposition}

\begin{proof}
First, the eigenvectors of $H_f$ are the canonical basis states. Indeed, each term of the sum that constitutes $H_f$ is a tensor product of $n$ matrices $I$ or $Z$, both diagonal. Thus, $H_f$ is a $2^n$ diagonal matrix.
Second, let us find the eigenvalues associated with the eigenvectors. Let $\ket{x} = \ket{x_1 \ldots x_n}$ be in $\CB$. Let $z=(z_1,\ldots,z_n)$ be the result of transformation~\eqref{eq:transfo}. Thus, we can easily show that $Z\ket{x_i} = z_i\ket{x_i}\,,\forall i\in[n]$ and
\begin{align*}
    H_f\ket{x} &= \mathlarger\sum_{\alpha = (\alpha_1,\ldots,\alpha_n) \in \{0,1\}^n} h_{\alpha}\bigotimes_{i=1}^n Z_i^{\alpha_i}\ket{x_i}\\
    &= \mathlarger\sum_{\alpha = (\alpha_1,\ldots,\alpha_n) \in \{0,1\}^n} h_{\alpha}\bigotimes_{i=1}^n z_i^{\alpha_i}\ket{x_i}\\
    &= \left(\mathlarger\sum_{\alpha = (\alpha_1,\ldots,\alpha_n) \in \{0,1\}^n} h_{\alpha}\prod_{i=1}^nz_i^{\alpha_i}\right)\ket{x}\\
    &= f_\pm(z)\ket{x}\\
    &= f(x)\ket{x}\,.
\end{align*}
\end{proof}

\begin{example}
\label{ex:prob_example}
We illustrate this transformation on a small example with $n=2$. Let us consider the problem
\begin{equation*}
    \min_{x\in\{0,1\}^2} f(x) = x_1 + 2x_2 -3x_1x_2\,.
\end{equation*}
Using~\eqref{eq:transfo}, the equivalent $\{-1,1\}$-problem is
\begin{equation*}
    \min_{z\in\{-1,1\}^2} f_{\pm}(z) = \frac{1}{4}z_1 - \frac{1}{4}z_2 - \frac{3}{4}z_1z_2 + \frac{3}{4}\,.
\end{equation*}
Thus, the Hermitian matrix associated with the problem is
\begin{equation*}
    H_f = \frac{1}{4}Z\otimes I - \frac{1}{4}I\otimes Z - \frac{3}{4}Z\otimes Z +\frac{3}{4}I \otimes I\,.
\end{equation*}
To illustrate~\eqref{eq:eigenvectors}, we compute the eigenvalue of the canonical basis state $\ket{10}$.
\begin{align*}
    H_f\ket{10} &=  \frac{1}{4}(Z\otimes I)\ket{10} - \frac{1}{4}(I\otimes Z)\ket{10} - \frac{3}{4}(Z\otimes Z)\ket{10} +\frac{3}{4}(I \otimes I)\ket{10}\\
    &=  -\frac{1}{4}\ket{10} - \frac{1}{4}\ket{10} + \frac{3}{4}\ket{10} +\frac{3}{4}\ket{10}\\
    &= \ket{10}\\
    &= f(1,0)\ket{10}\,,
\end{align*}
because $f(1,0) = 1$.
\end{example}

Notice that most of the problems solved with QAOA in the literature are QUBO (Quadratic Unconstrained Binary Optimization) problems. Thus, $H_f$ has the specific form
\begin{equation*}
    H_f = \sum_{i}h_{ii}Z_i + \sum_{i<j}h_{ij} Z_i\otimes Z_j\,,
\end{equation*}
where $h_{ij}\in\mathbb{R},\,\forall i\leq j$. 
It is justified by the fact that solving QUBO problems to optimality is \emph{already} NP-hard, and the quantum gates of the circuit are easier to implement on hardware in that case. 

\subsection{Quantum part}
\label{def:QAOA_circuit}
QAOA is a Variational Quantum Algorithm where the quantum part derives from the Hamiltonian $H_f$. This quantum part consists of a quantum circuit with $2p$ parameters $(\vb*{\gamma},\vb*{\beta}) = (\gamma_1,\ldots,\gamma_p, \beta_1,\ldots,\beta_p) \in \mathbb{R}^{2p}$, where $p$ is called \emph{depth}. The quantum circuit $U(\vb*{\gamma},\vb*{\beta})$ is the sequence of $p$ layers of two blocks, initially applied to the uniform superposition $\ket{+}^{\otimes n}$ (see Example~\ref{ex:unif_superposition}). The first block is of the form $\Exp(H_f, \gamma)$, for $\gamma \in \mathbb{R}$, which is the unitary operator (see Definition \ref{unitmatrixfromH}) associated with the Hamiltonian $H_f$. The second block is of the form $\Exp(H_B, \beta)$, for $\beta \in \mathbb{R}$, which is the unitary operator associated with the Hamiltonian $H_B = \mathlarger\sum_{i=1}^{n}R_{X,i}(\pi)$. 
Thus, the quantum circuit is
\begin{equation}
\label{eq:quantum_circuit}
        U(\vb*{\gamma},\vb*{\beta}) = \Exp(H_B, \beta_p)\Exp(H_f, \gamma_p)\ldots\Exp(H_B, \beta_1)\Exp(H_f, \gamma_1)H^{\otimes n}\,,
\end{equation}
where the first gates applied in this circuit are $H^{\otimes n}$ to then apply the $p$ layers to state $\ket{+}^{\otimes n}$.
The three propositions that follow express the quantum circuit of QAOA with the set of universal gates (see Theorem~\ref{th:universality}). 
We give first the general decomposition of the QAOA quantum circuit defined in~\eqref{eq:quantum_circuit}, see Appendix~\ref{proof:Exp_general_circuit} for a proof.
\begin{restatable}{proposition}{Expgeneralcircuit}\label{prop:Exp_general_circuit}
The first block $\Exp(H_f, \gamma)$ parametrized by $\gamma\in\mathbb{R}$ is
\begin{equation*}
    \Exp(H_f, \gamma) = \prod_{\alpha\in\{0,1\}^n}\Exp(\bigotimes_{i=1}^n Z_i^{\alpha_i}, h_\alpha \gamma)\,.
\end{equation*}
The second block $\Exp(H_B, \beta)$ parametrized by $\beta\in\mathbb{R}$ is
\begin{equation*}
    \Exp(H_B, \beta) = \bigotimes_{i=1}^n R_{X,i}(2\beta)\,.
\end{equation*}
\end{restatable}

The particular case of QUBO is mainly considered in the literature. Thus, we propose next a decomposition of the QAOA quantum circuit for this specific case, see Appendix~\ref{proof:ExpQUBO} for a proof.

\begin{restatable}{proposition}{ExpQUBO}\label{prop:ExpQUBO}
For the case of QUBO, the expression of $\Exp(H_f, \gamma)$ simplifies in
\begin{equation*}
    \Exp(H_f, \gamma) = \left( \bigotimes_{i=1}^n R_{Z,i}(2h_{ii}\gamma)\right) \prod_{i<j}\CX_{i,j}  R_{Z,j}(2h_{i,j}\gamma)\CX_{i,j}\,,
\end{equation*}
and is rather easily implemented with universal quantum gates.
\end{restatable}

The decomposition in universal quantum gates of the term $\Exp(Z_i\otimes Z_j, t)$ for the specific case of QUBO (see Proposition~\ref{prop:ExpQUBO}) is mainly used in the literature. We propose in Proposition~\ref{prop:ExpGeneral} a generalization of such a decomposition for the term $\Exp(\bigotimes_{i=1}^n Z^{\alpha_i}, t)$, where the number of Z gates in effect can be bigger than two, namely, $|\{\alpha_i,\, i\in[n] : \alpha_i = 1\}| \geq 2$. This proposition enables us to overtake the QUBO problems, namely, to deal with a polynomial function $f$ with a degree strictly larger than two. 
We introduce next a technical result that will be necessary to derive the subsequent proposition, see Appendix~\ref{proof:induction} for a proof. 
\begin{restatable}{lemma}{induction}\label{lem:induction}
$\forall n\in\mathbb{N}^*$,
    $$ (I^{\otimes n-1}\otimes X)e^{-itZ^{\otimes n}}(I^{\otimes n-1}\otimes X) = e^{itZ^{\otimes n}}\,.$$
\end{restatable}

The proposition that follows enables a decomposition in universal quantum gates of any quantum circuit of QAOA, see Appendix~\ref{proof:ExpGeneral} for a proof. 
\begin{restatable}{proposition}{ExpGeneral}\label{prop:ExpGeneral}
Let us consider the subsystem composed of the $N$ qubits to which the $Z$ gate is applied. Specifically, $N = |\{\alpha_i\, i\in[n] : \alpha_i = 1\}|$, and we renumber the qubits in question in $[N]$.
Thus, for $N\geq 2$, the term $\Exp(\bigotimes_{i=1}^n Z^{\alpha_i}, t)$ on this subsystem simplifies in
$$\Exp(Z^{\otimes N}, t) = \prod_{j=0}^{N-2} \CX_{1,N-j} R_{Z,N}(2t) \prod_{j=0}^{N-2} \CX_{1,N-j}\,.$$
We represent this decomposition on Figure~\ref{fig:ExpGeneral}.
\begin{figure}[H]
$$
  \Qcircuit @C=1em @R=.7em {
& \ctrl{3} & \qw & \qw & \qw & \qw & \qw & \qw & \qw& \qw & \qw& \qw & \qw & \qw & \ctrl{3} &  \qw \\
& \qw &  \ctrl{2} & \qw & \qw & \qw & \qw & \qw & \qw & \qw &  \qw& \qw & \qw &\ctrl{2} & \qw  & \qw \\
& \qw & \qw & \qw & \qddots &  & \ctrl{1} & \qw & \ctrl{1} & \qw & \reflectbox{\qddots} & & \qw & \qw & \qw & \qw\\
& \gate{X} &  \gate{X} & \qw & \qldots & &  \gate{X} & \gate{R_Z(2t)} & \gate{X} & \qw &\qldots & & \qw &\gate{X} & \gate{X} & \qw
}
$$
    \caption{Decomposition of $\Exp(Z^{\otimes N}, t)$ on the $N$-qubit subsystem.}
    \label{fig:ExpGeneral}
\end{figure}
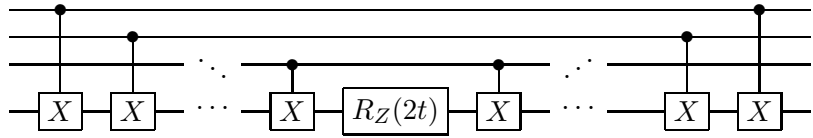
\end{restatable}

\begin{example}
We illustrate the construction of the quantum circuit with the problem of Example~\ref{ex:prob_example} where $n=2$, for the case $p=1$. Thus, 
\begin{equation*}
    U(\gamma, \beta) = \Exp(H_B, \beta)\Exp(H_f, \gamma)\ket{+}^{\otimes 2},~~\forall \gamma, \beta\in\mathbb{R}\,.
\end{equation*}
The circuit is detailed in Figure~\ref{fig:circuit_QAOA}.
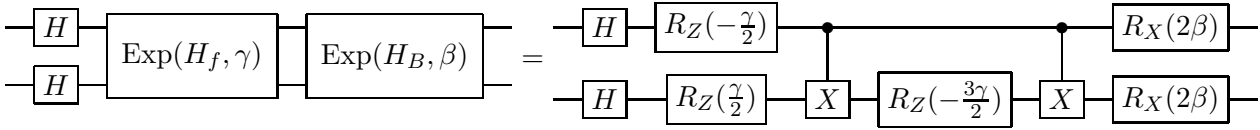
\begin{figure}[H]
    \centering
    \begin{align*}
\Qcircuit @C=1em @R=.7em {
& \gate{H} & \multigate{1}{\Exp(H_f, \gamma)} & \multigate{1}{\Exp(H_B, \beta)} & \qw \\
& \gate{H} & \ghost{\Exp(H_f, \gamma)} & \ghost{\Exp(H_B, \beta)} & \qw 
} 
\text{  }\raisebox{-1.1em}{=}\text{  }
    \Qcircuit @C=1em @R=.7em {
& \gate{H} & \gate{R_Z(-\frac{\gamma}{2})} & \ctrl{1} &  \qw & \ctrl{1} & \gate{R_X(2\beta)} & \qw\\
& \gate{H} & \gate{R_Z(\frac{\gamma}{2})} & \gate{X} &  \gate{R_Z(-\frac{3\gamma}{2})} & \gate{X}& \gate{R_X(2\beta)} & \qw
} 
\end{align*}
    \caption{QAOA circuit of Example~\ref{ex:prob_example} for $p=1$.}
    \label{fig:circuit_QAOA}
\end{figure}
Notice that we do not take into account the term $\frac{3}{4}I\otimes I$ of $H_f$ in this circuit. More generally, the term $h_{0\ldots0}I^{\otimes n}$ of $H_f$ never appears on the quantum circuit because it represents a constant term and does not influence the optimization. 
\end{example}

Notice that the choice of QAOA circuit does not ensure~\eqref{eq:incl_U}. For example, one can show that the probability of measuring 00 at the end of the circuit depicted in Figure~\ref{fig:circuit_QAOA} never reaches 1. Specifically,
\begin{equation*}
    p_{(\gamma, \beta)}(00) = |\bra{00}\Exp(H_B, \beta)\Exp(H_f, \gamma)\ket{+}^{\otimes 2}|^2 < \frac{1}{2},~~\forall \gamma, \beta\in\mathbb{R}\,.
\end{equation*}
However, QAOA satisfies another important property: it uses entangling gates. Each gate $\Exp(Z_i\otimes Z_j, t)$, for $t\in\mathbb{R}\setminus \{k\pi : k\in\mathbb{Z}\}$, entangles the qubits $i$ and $j$. Indeed, $\Exp(Z_i\otimes Z_j, t) = \CX_{i,j}  R_{Z,j}(2t)\CX_{i,j}$, and unless $R_{Z,j}(2t) = I$, namely $t\in\{k\pi : k\in\mathbb{Z}\}$, the CNOT gate operates\footnote{For $t\in\{k\pi : k\in\mathbb{Z}\},\,R_{Z,j}(2t) = I$. Thus, $\CX_{i,j}  R_{Z,j}(2t)\CX_{i,j} = \CX_{i,j}\CX_{i,j} = I$ because CNOT is its own inverse.} and creates entanglement as mentioned in Subsection~\ref{sec:entanglement}. The other gates, which are one-qubit gates, do not have this power. Entanglement is not necessary to~\eqref{eq:incl_U}. However, Remark~\ref{rem:cont_relax} can justify the use of entanglement gate. Indeed, without it, it seems unlikely to achieve better results than pure classical optimization because the optimizer essentially solves a classical continuous relaxation.

In fact, the popularity of QAOA originates essentially from the fact that it mimics the adiabatic schedule~\cite{Eq_Adiab_QAOA}. The $n$-qubit system verifies the adiabatic condition when $p\rightarrow\infty$ and ensures that for a particular set of parameters, the quantum circuit gives the exact solution. However, quantum computers' quality today makes implementations on large instances impossible. But because the number of gates of the quantum circuit is $\mathcal{O}(pn^2)$, for small depth $d$, QAOA, and more generally VQAs, can already be implemented on current NISQ computers. 

\section{Literature review for QAOA}
Many papers have recently addressed the empirical evaluation of QAOA, some also comparing it with specific implementations of VQA. We present below a non-exhaustive list of these trends, we mention several theoretical limitations for specific cases known up to now and we end with the different leverages that are at stake to improve QAOA performances. This section does not provide a complete, up-to-date overview of QAOA performances, but rather aims at illustrating trends on combinatorial problems of interest to the operations research community.

\subsection{Empirical and theoretical trends on QAOA}

Let us begin with the numerical trends of QAOA performances. All the empirical experiments are presented on small instances because quantum computers' quality today makes implementations on large instances impossible, leading to difficult conclusions. Thus, many experiments are done on classical simulators of quantum computers.  
Most of the empirical results of QAOA apply to the MAX-CUT problem because it was initially the first application of QAOA~\cite{QAOAFarhi}. 
\begin{definition}[MAX-CUT problem]
Let $G=(V,E)$ be a undirected graph. A cut in $G$ is a subset $S\subseteq V$. We define its \emph{cost} as the number of edges with one node in $S$ and one node in $V\setminus S$. The MAX-CUT problem aims at finding a cut with maximum \emph{cost}. A version with weighted edges can also be defined. 
\end{definition}
Notice that for the MAX-CUT problem, with the notations of Section~\ref{sect:QAOA}, the objective function is
\begin{equation*}
    f(x) = -\sum_{(i,j)\in E} \left(x_i(1 - x_j) + x_j(1-x_i)\right)\,,
\end{equation*}
where $x_i$ is 1 if node $i$ is in the cut, 0 otherwise. The Hermitian matrix that corresponds to this problem is
\begin{equation*}
    H_f = -\frac{1}{2}\sum_{(i,j)\in E} (1 - Z_i Z_j)\,.
\end{equation*}
The approximation ratio $r$ mainly quantifies the performance of QAOA as follows:
\begin{equation*}
    r = \frac{f^\text{QAOA}}{f^*}\,,
\end{equation*}
where $f^\text{QAOA}$ is the value returned by QAOA, and $f^*$ is the optimal value. Papers often compare this ratio with the best-known guaranteed ratio of Goemans-Williamson algorithm~\cite{GW}, specifically, $r=0.87856$. The seminal paper of QAOA~\cite{QAOAFarhi} provides a lower bound of $r$ for the specific class of 3-regular graphs for $p=1$ that is $r=0.6924$. More precise analytical expressions of the lower bound of the ratio for $p=1$ and for some other typical cases are given in~\cite{Hadfield_fermionic}.

Several empirical results on MAX-CUT spotlight patterns of optimal parameters and enable QAOA to exceed Goemans-Williamson bound for some specific instances. Some classes of MAX-CUT instances studied reveal patterns of optimal parameters. Thus, it seems to offer efficient heuristics for parameter selection and initialization. For example, the authors of~\cite{Crooks} look at the class of Erdös-Rényi graphs (random graphs where an edge appears between two nodes with probability 0.5) of size up to $17$ nodes, where the classical optimizer is an automatic differentiation with stochastic gradient descent. The authors of~\cite{Empirical_bound} examine the exhaustive set of graphs with $n\leq9$ nodes, with the gradient-based search BFGS (Broyden–Fletcher–Goldfarb–Shanno algorithm~\cite{BFGS}) for the classical optimizer. Both exhibit instances that exceed the bound of Goemans-Williamson for small depth, $p\leq 8$ and $p\leq 3$, respectively. The sets of unweighted and weighted 3-regular graphs also lead to patterns in~\cite{Zhou} for graphs of a maximum size of 22 nodes, also detected in the parameter space of the Job Shop Scheduling problem~\cite{JobShop}. But even if the performance of QAOA sometimes exceeds the Goemans-Williamson bound for the low-depth circuits, it is believed that $p$ must grow with the instance size to have a chance to outperform the best classical algorithms.

Indeed, some theoretical limits of QAOA are displayed, where the shape of the quantum state produced by the quantum circuit is at stake. The authors of \cite{Bravyi} point out that the symmetry and locality of this resulting variational state fundamentally limit the performances of QAOA. Indeed, they show that Goemans-Williamson outperforms QAOA for several instances of the MAX-CUT problem for any fixed depth $p$. Consequently, this paper suggests a non-local version of QAOA to overcome these limitations. The limits of the locality also appear when solving the problem of the Maximum Independent Set (MIS).  In~\cite{SeeAllGraph}, MIS instances are random graphs of $n$ vertices, with a fixed average degree $\bar{d}$. Thus, it proves that for depth $p\leq C\log(n)$, where $C$ is a constant depending on $\bar{d}$, QAOA cannot return an independent set better than 0.854 times the optimal for $\bar{d}$ large. Due to this locality issue, the author of~\cite{Hastings} compares QAOA with \emph{local} classical algorithms which also have this locality notion: at each step, the value of a variable is updated depending on the values of its \emph{neighbors}, i.e., the variables that share the same term within the objective function. Namely, after $t$ steps, the value of a variable depends on all information gathered in its $t$-\emph{neighborhood}. Yet, these classical algorithms still outperform QAOA. A single step of these algorithms outperforms, resp. achieves the same performance as, a single step of QAOA ($p=1$) for MAX-CUT  instances, resp. MAX-3-LIN-2 instances. Notice that given binary variables and a set of linear equations modulo 2 with exactly three variables, the MAX-3-LIN-2 problem aims at finding a variable value assignment that maximizes the number of satisfied equations.
Several other canonical combinatorial optimization problems have been tackled with QAOA in recent years. While these examples are limited to numerical tests on toy instances and do not include comparisons of performance with classical resolutions, we mention some of them to illustrate the growing interest in solving reference combinatorial problems with Variational Quantum Algorithms. For instance, in~\cite{TSP}, the authors reformulate the Traveling Salesman Problem (TSP) as an appropriate Hamiltonian matrix and test it on small instances of three and four cities. The authors of~\cite{GraphColoring} reformulate and solve the Graph Coloring problem for a dozen of nodes. The authors of~\cite{JobShop} study the Job Shop Scheduling problem, find a suitable formulation, and implement it on an artificial instance with three machines and three jobs, each of them containing one or two operations.
We end by citing a benchmark provided in~\cite{IBM_Benchmark}, that compares classical techniques, among them the Simulated Annealing, and quantum techniques including VQAs on NISQ quantum computers. It applies to the TSP and the Quadratic Assignment problem and shows that in terms of running time and quality solution, the classical methods significantly outperform the quantum ones. Notice that besides the latter paper, the mention of the computation time for VQAs in literature is rare. It could be explained by the fact that most of the experiments are tested on trivial instances and solved with quantum algorithms simulated on classical high-performance computers due to the noise on current quantum hardware~\cite{Noise_DePalma}.

\subsection{Improvements and adaptations of QAOA}
Despite the theoretical limitations displayed above, QAOA has leverages (guiding function, parametrized quantum circuit, classical optimizer, etc.) that are still of interest in the literature. We display some of these studies on different choices of leverages that empirically improve QAOA performances. 

First, the guiding function is mainly the mean function~\eqref{def:g_mean} as in the seminal paper of QAOA. However, both the CVaR function~\eqref{def:g_cvar} and the Gibbs function~\eqref{def:g_gibbs} give an alternative to the mean function and show empirical improvements. For the former, several optimization problems, such as MAX-CUT, Maximum Stable Set, MAX-3SAT, etc., are solved with QAOA in~\cite{CVaR} and show better results with faster convergence. For the latter, the authors of~\cite{Gibbs} display better results solving MAX-CUT with this guiding function. Notice that comparing these improvements is hard because they use different classical optimizers. A different method to guide the optimization, presented in~\cite{F-VQE} as a Filtering Variational Quantum Algorithm, substitutes the guiding function for \emph{filtering functions}. A \emph{filtering function} is a function associated with a \emph{filtering operator} that, given a Hamiltonian matrix and a quantum state, essentially modifies the latter by increasing the probability of eigenvectors with low eigenvalues and decreasing the probability of eigenvectors with high eigenvalues. Empirical results show improvements in the quality solution and the speed of convergence for the weighted Max-Cut problem.

Second, the choice of the quantum circuit is challenged in the literature. The underlying question is whether the quantum circuit of QAOA (see Subsection~\ref{def:QAOA_circuit}) is a good choice for finite depth $p$. Several papers suggest better circuits based on empirical results for small depth and small instances. For example, VQA with the circuit proposed in~\cite{CVaR} has better performances than QAOA, or with the \emph{Bang-Bang} circuit described in~\cite{BangBang}. The benefit of entanglement gates in the circuit is also discussed in~\cite{Nannicini_Hyb_Perf} without giving a clear advantage. Notice that the authors of~\cite{WarmStart} suggest to warm-start QAOA with either a continuous relaxation or a randomized rounding. This consists of initializing the quantum circuit with a solution of a continuous relaxation (Quadratic Programming or Semi Definite Programming), resp. with a randomly rounded solution of a continuous relaxation, instead of the state $\ket{+}^{\otimes n}$. In both cases, QAOA performances for small $p$ are better with a warm-start.

The choice of the classical optimizer represents another leverage, where both gradient-based and gradient-free optimizers can be used. The most encountered gradient-based methods used in VQAs in the literature are the Gradient Descent~\cite{GradientDescent}, the Broyden–Fletcher–Goldfarb-Shanno algorithm~\cite{BFGS}, and the stochastic optimization algorithm ADAM~\cite{ADAM}. As gradient-free methods, we can find the constrained optimization with linear approximation algorithm COBYLA~\cite{COBYLA}, the simplex-based algorithm Nelder-Mead~\cite{NelderMead}, and the simultaneous perturbation stochastic approximation algorithm SPSA~\cite{SPSA}. The authors of~\cite{Nannicini_Hyb_Perf} advise the choice of a global optimizer rather than a local optimizer to avoid numerous local optima. In~\cite{EstimDistrib_GECCO}, the authors choose the gradient-free evolutionary algorithm called Estimation of Distribution Algorithm as the classical optimizer. This algorithm generates new solutions from a probabilistic model that depends on the best solutions of the previous iterations. They empirically show that it improves the results compared to traditional optimizers as mentioned above.
The expression of the quantum circuit can also produce barren plateaus, hardening the optimization~\cite{BarrenPlateau, Sol_BarrenPlateau}. We do not list all the possible choices of VQA leverages. The survey~\cite{Cerezo} proposes other possibilities.

In parallel, several algorithms derived from QAOA are appearing in the literature. An adaptation of QAOA is the Recursive-QAOA~\cite{RQAOA}, also called RQAOA. It applies first several times QAOA to the problem in order to reduce its size, namely the number of variables, according to the correlation that appears between some variables. Then, it solves with classical brute force the resulting smaller problem. This algorithm seems to be competitive compared to QAOA for problems such as MAX-k-CUT. The authors of~\cite{Adaptative_QAOA} propose another algorithm, the Adaptive-QAOA, that converges faster than QAOA on some instances of MAX-CUT. It consists of applying recursively QAOA, increasing step by step the depth of the quantum circuit.

Eventually, there is ongoing work on the formulation and implementation of the QAOA circuit to ease and lighten QAOA implementation. For instance, the authors of~\cite{kSAT_hamiltoniancycles_GECCO} present an algorithmic method to reduce the growth of the QUBO matrix size with the problem size for the $k$-SAT problem and the Hamiltonian Cycles problem. 
Besides that, the authors of~\cite{Herrman} expose a \emph{global variable substitution method} that, given an initial linear formulation of a 3-SAT problem, exploits the advantage of product representation in QAOA and, thus, minimizes the circuit depth\footnote{Notice that here we talk about the general depth, not $p$, which is the longest path in the circuit, namely the maximum number of gates executed on a qubit.}. The shallower depth, the more efficient the implementation of the quantum circuit for NISQ computers. 
A different approach is proposed in~\cite{Nagarajan}, that is, given a quantum circuit, minimizes the circuit depth by solving a Mixed-Integer Program, with an optimality guarantee on the quantum circuit produced (it can find up to 57\% reduction of the number of gates). It applies to any gate-based quantum algorithm and thus, can be applied to QAOA and, more generally, to VQAs.

Notice that solving combinatorial optimization problems using QAOA involves first formulating them into unconstrained problems, and more precisely into QUBO for easier circuits. There are essentially two types of formulation. The first and most common one is to integrate constraints as suitable penalty terms into the objective function. For instance, in~\cite{Lucas} we find formulations of Karp's 21 NP-complete problems, and the authors of~\cite{kcoloring} tackle the $k$-coloring graph problem with the same method. The tutorial~\cite{Glover} addresses more general cases for this formulation problem methodology.
The second type of formulation is to change the expression of the \emph{mixing} Hamiltonian (referred to as $H_B$ in Section~\ref{sect:QAOA}). Initially presented in~\cite{Hadfield}, the main idea is to make this Hamiltonian varying the quantum state only in the feasible search space. 
Some problems have been specifically studied with this method, such as the Traveling Salesman problem~\cite{TSP_Ruan, TSP}.

\section{Conclusion}
In this tutorial, we have provided a mathematical description of VQAs and how they are translated into quantum algorithms. There has been a growing interest in VQAs, and more particularly QAOA, mostly because of their compatibility with NISQ computers. We have described how VQAs raise interesting questions regarding the choice of guiding functions and quantum circuits, among others. They also raise many technical questions that we did not address in this work, such as the noise at the implementation level~\cite{Lao2019MappingOQ}.

Since its inception, researchers have argued about the potential advantages of QAOA over classical optimization algorithms~\cite{Farhi_Quantum_Sup, Empirical_bound, MAXkXOR}.
This research seems to indicate that there is currently no scientific evidence that QAOA will soon beat classical heuristics. The available results even suggest the opposite, such as~\cite{Hastings} showing that local search algorithms can do better at low depth. Moreover, the current numerical results are hard to assess because they are run on small instances for which optimal or near-optimal solutions can be obtained easily with classical algorithms. Hopefully, the quickly growing capabilities of quantum computers will soon lead to a better understanding of the numerical efficiency of such algorithms~\cite{Guerreschi}.

\section*{Acknowledgements}

This work has been partially financed by the ANRT (Association Nationale de la Recherche et de la Technologie) through the PhD number 2021/0281 with CIFRE funds.

\bibliographystyle{abbrv}
\bibliography{ref}

\appendix
\section{Basics of linear algebra}
\label{sec:appendix}
This section provides some basic notions of linear algebra used throughout the tutorial. Notice that some specific definitions, such as unitary and Hermitian matrices, are introduced together with their use in quantum computing.

\begin{definition}[Tensor product]
\label{def:tensorProduct}
Let $A = \begin{pmatrix} 
    a_{11} & \dots  & a_{1m}\\
    \vdots & \ddots & \vdots\\
    a_{n1} & \dots  & a_{nm} 
    \end{pmatrix} \in \mathcal{M}_{n,m}(\C)$ and $B = \begin{pmatrix} 
    b_{11} & \dots  & b_{1q}\\
    \vdots & \ddots & \vdots\\
    b_{p1} & \dots  & b_{pq} 
    \end{pmatrix} \in \mathcal{M}_{p,q}(\C)$ be two complex matrices. Then, we define their tensor product, a bilinear operation, as 
    \begin{equation*}
    A \otimes B = \begin{pmatrix} 
    a_{11}B & \dots  & a_{1m}B\\
    \vdots & \ddots & \vdots\\
    a_{n1}B & \dots  & a_{nm}B 
    \end{pmatrix} \in \mathcal{M}_{np,mq}(\C)\,.
    \end{equation*}

where $a_{ij}B = \begin{pmatrix} 
    a_{ij}b_{11} & \dots  & a_{ij}b_{1q}\\
    \vdots & \ddots & \vdots\\
    a_{ij}b_{p1} & \dots  & a_{ij}b_{pq} 
    \end{pmatrix} \in \mathcal{M}_{p,q}(\C)$.
\end{definition} 

\begin{proposition}[Tensor product properties]
We draw attention to some useful properties of the tensor product : 
\begin{itemize}
    \item $A\otimes B \neq B\otimes A$
    \item $(A \otimes B) \otimes C = A \otimes (B \otimes C) = A \otimes B \otimes C$
    \item $A\otimes(cB) = (cA)\otimes B = c(A\otimes B)$, for $c\in\mathbb{C}$
    
\end{itemize}
\end{proposition}

\begin{definition}[Braket notation]
We also introduce the \emph{braket} notation used in quantum computing. Let $\psi = \begin{pmatrix}
    \psi_1\\
    \vdots \\
    \psi_{2^n}
\end{pmatrix} \in\C^{2^n}$ be a column vector. We note $\ketpsi$, said \emph{ket} $\psi$, the vector $\psi$ itself. Thus, we define $\bra{\psi}$, said \emph{bra} $\psi$, the conjugate transpose of $\ketpsi$. Specifically,
\begin{align*}
\ket{\psi} &= \psi = \begin{pmatrix}
    \psi_1 \\
    \vdots \\
    \psi_{2^n}
\end{pmatrix}\,,\\
\bra{\psi} &= \bar{\psi}^T = \begin{pmatrix}
    \bar{\psi_1} & \ldots &  \bar{\psi_{2^n}}\\
\end{pmatrix}\,.
\label{eq:braket}
\end{align*}
Thus, \emph{ket} vectors are always column vectors whereas \emph{bra} vectors are row vectors. 
With this notation, the norm of the vector $\psi$ is
\begin{equation*}
    \braket{\psi} = \mathlarger\sum_{i=1}^{2^n} |\psi_i|^2\,.
\end{equation*} 
We also consider the complex inner product
\begin{equation*}
    \psi, \phi \mapsto \braket{\psi}{\phi}\,,
\end{equation*}
where we use the notation $\bra{\psi}\cdot\ket{\phi} = \braket{\psi}{\phi}$. Notice that it is not commutative. Indeed, for $\psi = \begin{pmatrix}
    \psi_1 \\
    \vdots \\
    \psi_{2^n}
\end{pmatrix} \in\C^{2^n}$ and $\phi = \begin{pmatrix}
    \phi_1 \\
    \vdots \\
    \phi_{2^n}
\end{pmatrix} \in \C^{2^n}$, we have : 
\begin{equation*}
    \begin{split}
        \braket{\phi}{\psi} &= \begin{pmatrix}
    \bar{\phi_1} & \ldots &  \bar{\phi_{2^n}}\\
\end{pmatrix}\begin{pmatrix}
    \psi_1 \\
    \vdots \\
    \psi_{2^n}
\end{pmatrix} \\
&= \mathlarger\sum_{i=1}^{2^n} \bar{\phi_i}\psi_i\\
& \neq \mathlarger\sum_{i=1}^{2^n} \phi_i\bar{\psi_i} 
= \braket{\psi}{\phi}\,.
    \end{split}
\end{equation*}
\end{definition}

\begin{definition}[Unitary operator associated with Hermitian matrix]
\label{unitmatrixfromH}
Given a Hermitian matrix $A$, we define its associated quantum gate $\Exp(A, t)$ parametrized by the parameter $t\in\mathbb{R}$ as :
\begin{align*}
    \Exp(A, t) &= e^{-iAt}\\
    &= \mathlarger\sum_{k=0}^{\infty} \frac{1}{k!}(-i)^kt^kA^k\,.
\end{align*}
Because $A$ is Hermitian, $\Exp(A,t)$ is a unitary matrix.
\end{definition}

\section{Omitted proofs}
\label{appendix:proofs}

\subsection{Guiding function properties}
\label{sec:proof:guiding_func}
\subsubsection{Proof of Proposition~\ref{prop:g_mean_guinding}}
\label{proof:g_mean_guinding}
\gmeanguinding*
\begin{proof}
Let us prove that $g_{\text{mean}}$ is continuous.
Let $x\in\{0,1\}^n$. The function $\theta \mapsto p_\theta(x) = |\bra{x}U(\theta)\ket{0_n}|^2$ is continuous, because each coefficient of $U(\theta)$ is continuous. Thus, because multiplication and addition preserve continuity, $g_{\text{mean}}$ is continuous.

We prove by contradiction that~\eqref{eq:incl_guiding} holds. Let $\theta\in\G$ and let us consider the quantum state $\ket{\psi(\theta)} = U(\theta)\ket{0_n}$. Its decomposition in the canonical basis is $\ket{\psi(\theta)} = \sum_{x \in \{0,1\}^n}\psi_x\ket{x}$.

Assume that $\ket{\psi(\theta)} \notin \Fq$. By definition, there exists $x_0\in \{0,1\}^n$ such that $x_0\notin \F$ and $|\psi_{x_0}|\neq 0$.
Thus, 
\begin{align*}
    g_{\text{mean}}(\theta) 
    &= \sum_{x \in \{0,1\}^n} |\psi_x|^2f(x)\\
    &= \sum_{x \in \F} |\psi_x|^2f(x) + \sum_{x \notin \F} |\psi_x|^2f(x)\\
    &= \left(\sum_{x \in \F} |\psi_x|^2\right)f^* + \sum_{x \notin \F} |\psi_x|^2f(x)\,,
\end{align*}
where $f^*$ is the optimal value of $f$, reached on $\F$. By definition, $f(x) > f^*\,, \forall x\notin \F$, and because we assume that $|\psi_{x_0}|\neq 0$, thus the second term of the sum is bounded below as follows:
$\sum_{x \notin \F} |\psi_x|^2f(x) > \left(\sum_{x \notin \F} |\psi_x|^2\right)f^*$, where $\sum_{x \notin \F} |\psi_x|^2 \geq |\psi_{x_0}|^2 > 0$.
Thus, 
\begin{align*}
    g_{\text{mean}}(\theta) 
    &> \left(\sum_{x \in \F} |\psi_x|^2\right)f^* + \left(\sum_{x \notin \F} |\psi_x|^2\right)f^*
    = f^*\,.
\end{align*}
This contradicts the previous statement that $\theta \in\G$, as one readily verifies that the minimum of $g_{\text{mean}}$ is $g_{\text{mean}}^* = f^*$.
\end{proof}

\subsubsection{Proof of Proposition~\ref{prop:g_gibbs_guiding}}
\label{proof:g_gibbs_guiding}
\ggibbsguiding*
\begin{proof}
Let us prove that $g_{G,\eta}$ is continuous.
Let $x\in\{0,1\}^n$. The function $\theta \mapsto p_\theta(x) = |\bra{x}U(\theta)\ket{0_n}|^2$ is continuous, because each coefficient of $U(\theta)$ is continuous. Thus, because multiplication, addition, and composition preserve continuity, $g_{G,\eta}$ is continuous.

Let $\eta>0$. Let us prove by contradiction that~\eqref{eq:incl_guiding} holds. As before, let $\theta\in\G$ and let us consider the quantum state $\ket{\psi(\theta)} = U(\theta)\ket{0_n}$. Its decomposition in the canonical basis is  $\ket{\psi(\theta)} = \sum_{x \in \{0,1\}^n}\psi_x\ket{x}$. Assume that $\ket{\psi(\theta)} \notin \Fq$. Thus, there exists $x_0\in \{0,1\}^n$ such that $x_0\notin \F$ and $|\psi_{x_0}|\neq 0$.
Thus, 
\begin{align*}
    g_{G,\eta}(\theta) 
    &= -\ln\left({\sum_{x\in\{0,1\}^n}|\psi_x|^2e^{-\eta f(x)}} \right)\\
    &= -\ln\left({\sum_{x \in \F} |\psi_x|^2e^{-\eta f(x)} + \sum_{x \notin \F} |\psi_x|^2e^{-\eta f(x)}} \right)\\
    &= -\ln\left({\left(\sum_{x \in \F} |\psi_x|^2\right)e^{-\eta f^*} + \sum_{x \notin \F} |\psi_x|^2e^{-\eta f(x)}}\right)\,,
\end{align*}
where $f^*$ is the optimal value of $f$, reached on $\F$. By definition, $f(x) > f^*\,, \forall x\notin \F$, and because $\eta>0$, we have $e^{-\eta f(x)} < e^{-\eta f^*}$. Moreover, we assume that $|\psi_{x_0}|\neq 0$, thus the second term of the sum in the logarithm is bounded above as follows:
$\sum_{x \notin \F} |\psi_x|^2e^{-\eta f(x)} < \left(\sum_{x \notin \F} |\psi_x|^2\right)e^{-\eta f^*}$, where $\sum_{x \notin \F} |\psi_x|^2 \geq |\psi_{x_0}|^2 > 0$.
Thus, because $y \mapsto -\ln(y)$ is a decreasing function,
\begin{align*}
    g_{G,\eta}(\theta)  
    &> -\ln\left({\left(\sum_{x \in \F} |\psi_x|^2\right)e^{-\eta f^*} + \left(\sum_{x \notin \F} |\psi_x|^2\right)e^{-\eta f^*}}\right)
    = \eta f^*\,.
\end{align*}
This contradicts the previous statement that $\theta \in\G$. Indeed, we can easily verify that the minimum of $g_{G,\eta}$ is $g_{G,\eta}^* = \eta f^*$.
\end{proof}

\subsubsection{Proof of Proposition~\ref{prop:g_CVaR}}
\label{proof:g_CVaR}
\gCVaR*
\begin{proof}
For the same reason as for $g_\text{mean}$, $g_{C,\alpha}$ is continuous.

Let $\alpha \in ]0,1[\,,\theta\in\G$ and let us consider the quantum state $\ket{\psi(\theta)} = U(\theta)\ket{0_n}$. Its decomposition in the canonical basis is $\ket{\psi(\theta)} = \sum_{x \in \{0,1\}^n}\psi_x\ket{x}$. 

Let us prove that either $\ket{\psi(\theta)} \in \Fq$ or that $\mathlarger\sum_{x\notin \F}|\psi_{x}|^2 =  \mathlarger\sum_{x=N_\F + 1}^{2^n} |\psi_{x_i}|^2 < 1 - \alpha$, where $N_\F$ is the index that delimits $\F$, specifically, $$\F = \{x_i, i\in[N_\F]\}\,.$$ 
There are two cases for $N_\alpha$ (we recall that $N_\alpha$ depends on $\theta$):
\begin{itemize}
    \item If $N_\alpha \leq N_\F$, thus \begin{align*}
        g_{C,\alpha}(\theta) &= \frac{1}{\mathlarger\sum_{i\in[N_\alpha]\subseteq[N_\F]}|\psi_{x_i}|^2}\sum_{i\in[N_\alpha]\subseteq[N_\F]}|\psi_{x_i}|^2f(x_i)\\ &= f^*\,,
    \end{align*}
    where $f^*$ is the optimal value of $f$, reached on $\F$. This contradicts~\eqref{eq:incl_guiding} but the probability of sampling non-optimal solutions when measuring $\ket{\psi(\theta)}$ is strictly lower than $1 - \alpha$. Indeed, by definition of $N_\alpha$, 
    $$ \mathlarger\sum_{i=1}^{N_\F}|\psi_{x_i}|^2 \geq \mathlarger\sum_{i=1}^{N_\alpha}|\psi_{x_i}|^2 \geq \alpha\,. $$ 
    \item Otherwise, $N_\alpha > N_\F$. Let us prove by contradiction that $\ket{\psi(\theta)} \in \Fq$.
    Assume that $\ket{\psi(\theta)} \notin \Fq$. Thus, there exists $k>N_\F$ such that $|\psi_k|\neq 0$. We can show that $g_{C,\alpha}(\theta) > f^*$ using essentially the same proof as that of Proposition~\ref{prop:g_mean_guinding} for $x_0 = \min\{k > N_\F : |\psi_k|\neq 0\}$.  
    This contradicts the statement that $\theta\in\G$, because the minimum of $g_{C,\alpha}$ is $g_{C,\alpha}^* = f^*$.
\end{itemize}
\end{proof}

\subsection{Quantum circuit of QAOA}
\label{sec:proof:quantum_circuit}
\subsubsection{Proof of Proposition~\ref{prop:Exp_general_circuit}}
\label{proof:Exp_general_circuit}
\Expgeneralcircuit*
\begin{proof}
Let $\gamma\in\mathbb{R}$ and let us consider the first block $\Exp(H_f, \gamma)$. By Definition~\ref{unitmatrixfromH}, and because each pair of matrices of the family $\{\bigotimes_{i=1}^n Z_i^{\alpha_i} : \alpha= (\alpha_1,\ldots,\alpha_n)\in\{0,1\}^n\}$ commutes two by two, 
\begin{align}
    \Exp(H_f, \gamma) &= e^{-i(\sum_{\alpha = (\alpha_1,\ldots,\alpha_n) \in \{0,1\}^n} h_{\alpha}\bigotimes_{i=1}^n Z_i^{\alpha_i})\gamma}\label{eq:prod_exp1}\\
    &= \prod_{\alpha = (\alpha_1,\ldots,\alpha_n)\in \{0,1\}^n} e^{-ih_{\alpha}\bigotimes_{i=1}^n Z_i^{\alpha_i}\gamma}\\
    &= \prod_{\alpha = (\alpha_1,\ldots,\alpha_n)\in \{0,1\}^n}\Exp(\bigotimes_{i=1}^n Z_i^{\alpha_i}, h_{\alpha}\gamma)
    \label{eq:prod_exp2}\,.
\end{align}
Let $\beta\in\mathbb{R}$ and let us consider the second block $\Exp(H_B, \beta)$. For more readability, we write $X_i=R_{X,i}(\pi)$ the application of matrix $X=\begin{pmatrix}
    0&1\\1&0
\end{pmatrix}$ on qubit $i$. With the same development as above, and because each pair of matrices of the family $\{X_i: i\in[n]\}$ commutes two by two, 
\begin{equation*}
    \Exp(H_B, \beta) = \prod_{i=1}^n\Exp(X_i, \beta)\,.
\end{equation*}
Let $i\in[n]$. Thus, 
\begin{align}
\Exp(X_i,\beta) &=
    \mathlarger\sum_{k=0}^{\infty} \frac{1}{k!}(-i)^k\beta^kX_i^k\\
    &= \mathlarger\sum_{k=0}^{\infty} \frac{1}{(2k)!}(-i)^{2k}\beta^ {2k}X_i^{2k} + \mathlarger\sum_{k=0}^{\infty} \frac{1}{(2k+1)!}(-i)^{(2k+1)}\beta^ {(2k+1)}X_i^{(2k+1)}\\
    \label{eq:proofRX1}
    &= \mathlarger\sum_{k=0}^{\infty} \frac{1}{(2k)!}(-i)^{2k}\beta^ {2k}I + \mathlarger\sum_{k=0}^{\infty} \frac{1}{(2k+1)!}(-i)^{(2k+1)}\beta^ {(2k+1)}X_i\\
    \label{eq:proofRX2}
    &= \mathlarger\sum_{k=0}^{\infty} \frac{1}{(2k)!}(-1)^{k}\beta^ {2k}I -i \mathlarger\sum_{k=0}^{\infty} \frac{1}{(2k+1)!}(-1)^{k}\beta^ {(2k+1)}X_i\\
    &= \cos(t)I -i\sin(t)X_i\label{eq:proofRX4}\\
    &= R_{X,i}(2\beta)\,,
\end{align}
where line~\eqref{eq:proofRX1} exploits the fact that $X^2 = I$, implying $X^{2k} = I$ and $X^{2k+1} = X$. Moreover, line~\eqref{eq:proofRX2} applies the definition of the complex number $i$, and one can recognize the power series of cosinus and sinus functions. 
\end{proof}

\subsubsection{Proof of Proposition~\ref{prop:ExpQUBO}}
\label{proof:ExpQUBO}
\ExpQUBO*
\begin{proof}
Let $\gamma\in\mathbb{R}$. 
The application of~\eqref{eq:prod_exp1}--\eqref{eq:prod_exp2} to the case of QUBO gives
\begin{align*}
    \Exp(H_f, \gamma) &= e^{-i(\sum_{i=1}^n h_{ii}Z_i + \sum_{i<j}h_{ij} Z_i\otimes Z_j)\gamma}\\
    &= \prod_{i=1}^n e^{-iZ_i h_{ii}\gamma}\prod_{i<j}e^{-i Z_i\otimes Z_j h_{ij}\gamma}\\
    &= \prod_{i=1}^n \Exp(Z_i, h_{ii}\gamma)\prod_{i<j}\Exp(Z_i\otimes Z_j, h_{ij}\gamma)\,.
\end{align*}
Then, let us prove that, for $i\in [n]$ and $t\in\mathbb{R}$, $\Exp(Z_i, t) = R_{Z,i}(2t)$. The same development as above~\eqref{eq:proofRX1}--\eqref{eq:proofRX4}, replacing $X$ by $Z$ that have the same property $Z^2 = I$, gives
\begin{align}
\Exp(Z_i,t)
    &= \cos(t)I -i\sin(t)Z_i \label{eq:proofRZ3}\\
    &= R_{Z,i}(2t)
    \label{eq:proofRZ4}\,.
\end{align}
Eventually, line~\eqref{eq:proofRZ3} is the application of the gate
$\begin{pmatrix}
    e^{-it} & 0\\
    0 & e^{it} 
\end{pmatrix} = R_Z(2t)$ on qubit $i$, and the identity on the others. 

It remains to prove that, for $i<j \in [n]$ and $t\in\mathbb{R}$, $\Exp(Z_i\otimes Z_j, t) = \CX_{i,j}  R_{Z,j}(2t)\CX_{i,j}$. Following the same developments as above, we have 
\begin{equation}
\label{proofRZZ}
    \Exp(Z_i\otimes Z_j, t) =  \cos(t)I -i\sin(t)Z_i\otimes Z_j\,.
\end{equation}
We consider the two-qubit system that corresponds to the qubit $i$ as the first qubit and the qubit $j$ as the second qubit (the others are unchanged by the transformation). Thus, it remains to prove the equality of the two circuits depicted on Figure~\ref{fig:equ_circuit}.
\begin{figure}[H]
    \centering
    \begin{align*}
\Qcircuit @C=1em @R=.7em {
 & \multigate{1}{\Exp(Z_1\otimes Z_2, t)} & \qw \\
 & \ghost{\Exp(Z_1\otimes Z_2, t)} & \qw 
} \text{  }\raisebox{-.9em}{=}\text{  }
    \Qcircuit @C=1em @R=.7em {
 & \ctrl{1} &  \qw & \ctrl{1} & \qw\\
 & \gate{X} &  \gate{R_Z(2t)} & \gate{X}&  \qw
} 
\end{align*}
    \caption{Decomposition of $\Exp(Z_1\otimes Z_2, t)$ into universal gates.}
    \label{fig:equ_circuit}
\end{figure}
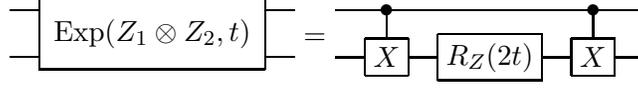

On the one hand,~\eqref{proofRZZ} is the application of the gate $ = \begin{pmatrix}
    R_Z(2t) & 0\\
    0 & R_Z(-2t)
\end{pmatrix}
$ to this system. Indeed,
\begin{align*}
   \cos(t)I -i\sin(t)Z_1\otimes Z_2 &= \cos(t)\begin{pmatrix}
    1 & 0 & 0 & 0 \\
    0 & 1 & 0 & 0 \\
    0 & 0 & 1& 0 \\
    0 & 0 & 0 & 1
\end{pmatrix} -i\sin(t)\begin{pmatrix}
    1 & 0 & 0 & 0 \\
    0 & -1 & 0 & 0 \\
    0 & 0 & -1& 0 \\
    0 & 0 & 0 & 1
\end{pmatrix} \\
&= \begin{pmatrix}
    e^{-it} & 0 & 0 & 0 \\
    0 & e^{it} & 0 & 0 \\
    0 & 0 & e^{it} & 0 \\
    0 & 0 & 0 & e^{-it} 
\end{pmatrix} \\
&= \begin{pmatrix}
    R_Z(2t) & 0\\
    0 & R_Z(-2t)
\end{pmatrix}\,.
\end{align*}
On the other hand, the composition of gates $\CX_{1,2}  R_{Z,2}(2t)\CX_{1,2}$ on this system amounts to
\begin{align*}
    \CX_{1,2}R_{Z,2}(2t)\CX_{1,2} &= \begin{pmatrix}
        I & 0 \\
        0 & X
    \end{pmatrix}\begin{pmatrix}
        R_Z(2t) & 0 \\
        0 & R_Z(2t)
    \end{pmatrix}\begin{pmatrix}
        I & 0 \\
        0 & X
    \end{pmatrix}\\
    &= \begin{pmatrix}
        R_Z(2t) & 0 \\
        0 & XR_Z(2t)X
    \end{pmatrix}\\
    &= \begin{pmatrix}
        R_Z(2t) & 0 \\
        0 & R_Z(-2t)
    \end{pmatrix}\,.
\end{align*}
Thus, the proof results from replacing $t$ by appropriate values $h_{ii}\gamma$ for $i\in[n]$ and $h_{ij}\gamma$ for $i< j$. 
\end{proof}

\subsubsection{Proof of Lemma~\ref{lem:induction}}
\label{proof:induction}
\induction*
\begin{proof}
Let $P_n$ be the statement
\begin{equation*}
(I^{\otimes n-1}\otimes X)e^{-itZ^{\otimes n}}(I^{\otimes n-1}\otimes X) = e^{itZ^{\otimes n}}\,.    
\end{equation*}
Let us prove by induction that $P_n$ holds for all integer $n\in\mathbb{N}^*$.\\
\textbf{Base case:} Let us prove $P_1$.
According to~\eqref{eq:proofRZ4}, $e^{-itZ} = R_Z(2t)$, thus, 
\begin{align*}
    Xe^{-itZ}X &= \begin{pmatrix}
        0&1\\1&0
    \end{pmatrix}\begin{pmatrix}
        e^{-it}&0\\0&e^{it}
    \end{pmatrix}\begin{pmatrix}
        0&1\\1&0
    \end{pmatrix}\\
    &= \begin{pmatrix}
        e^{it}&0\\0&e^{-it}
    \end{pmatrix}\\
    &= e^{itZ}\,.
\end{align*}
\textbf{Induction step:} Let $n\geq 1$ be given and suppose $P_n$. Let us prove $P_{n+1}$.
According to~\eqref{eq:decomp_Z(n+1)}, we have $
    e^{-itZ^{\otimes n+1}} = \begin{pmatrix}
    e^{-itZ^{\otimes n}}&0\\0 & e^{itZ^{\otimes n}}
\end{pmatrix}$. Thus,
\begin{align*}
    (I^{\otimes n}\otimes X)e^{-itZ^{\otimes n+1}}(I^{\otimes n}\otimes X) &= (I\otimes I^{\otimes n-1}\otimes X)e^{-itZ^{\otimes n+1}}(I\otimes I^{\otimes n-1}\otimes X)\\
    &= \begin{pmatrix}
        I^{\otimes n-1}\otimes X&0\\0&I^{\otimes n-1}\otimes X
    \end{pmatrix}\begin{pmatrix}
    e^{-itZ^{\otimes n}}&0\\0 & e^{itZ^{\otimes n}}
    \end{pmatrix}
    \begin{pmatrix}
        I^{\otimes n-1}\otimes X&0\\0&I^{\otimes n-1}\otimes X
    \end{pmatrix}\\
    &= \begin{pmatrix}
        (I^{\otimes n-1}\otimes X)e^{-itZ^{\otimes n}}(I^{\otimes n-1}\otimes X)&0\\0&(I^{\otimes n-1}\otimes X)e^{itZ^{\otimes n}}(I^{\otimes n-1}\otimes X)
    \end{pmatrix}\,.
\end{align*}
By induction hypothesis $P_n$, 
\begin{align*}
    (I^{\otimes n}\otimes X)e^{-itZ^{\otimes n+1}}(I^{\otimes n}\otimes X) 
    &= \begin{pmatrix}
        e^{itZ^{\otimes n}}&0\\0&e^{-itZ^{\otimes n}}
    \end{pmatrix}\\
    &= e^{itZ^{\otimes n+1}}\,.
\end{align*}
\end{proof}

\subsubsection{Proof of Proposition~\ref{prop:ExpGeneral}}
\label{proof:ExpGeneral}
\ExpGeneral*
\begin{proof}
Let $P_n$ be the statement
\begin{equation*}
\Exp(Z^{\otimes n}, t) = \prod_{j=0}^{n-2} \CX_{1,n-j} R_{Z,n}(2t) \prod_{j=0}^{n-2} \CX_{1,n-j}\,.    
\end{equation*}
Let us prove by induction that $P_n$ holds for all integer $n\geq 2$.\\
\textbf{Base case:} Proposition~\ref{prop:ExpQUBO} proves $P_2$. \\
\textbf{Induction step:} Let $n\geq 2$ be given and suppose $P_n$. Let us prove $P_{n+1}$. \\
On the one hand, $$\Exp(Z^{\otimes n+1}, t) = e^{-itZ^{\otimes n+1}} = e^{-itZ\otimes Z^{\otimes n}}$$ where $Z\otimes Z^{\otimes n} = \begin{pmatrix}
    Z^{\otimes n}&0\\0 & -Z^{\otimes n}
\end{pmatrix}$. This latter matrix is diagonal, thus, 
\begin{equation}
\label{eq:decomp_Z(n+1)}
    \Exp(Z^{\otimes n+1}, t) = \begin{pmatrix}
    e^{-itZ^{\otimes n}}&0\\0 & e^{itZ^{\otimes n}}
\end{pmatrix}\,.
\end{equation}
On the other hand, we compute the term
\begin{equation*}
\prod_{j=0}^{n-1} \CX_{1,n+1-j} R_{Z,n+1}(2t) \prod_{j=0}^{n-1} \CX_{1,n+1-j} =   \CX_{1,n+1}\left(\prod_{j=1}^{n-2} \CX_{1,n+1-j} R_{Z,n+1}(2t) \prod_{j=1}^{n-2} \CX_{1,n+1-j}\right) \CX_{1, n+1}
\end{equation*}
that is represented on Figure~\ref{fig:circuit_exp_proof}, where $e^{-itZ^{\otimes n}}$ applies on the qubits $2$ to $n$ by induction hypothesis. 
\begin{figure}[H]
$$
    \Qcircuit @C=1em @R=.7em {
& \ctrl{3} & \qw & \ctrl{3} & \qw \\
& \qw & \multigate{2}{e^{-itZ^{\otimes n}}} & \qw & \qw \\
& \qw & \ghost{e^{-itZ^{\otimes n}}} & \qw & \qw\\
& \gate{X} & \ghost{e^{-itZ^{\otimes n}}} & \gate{X} & \qw
}
$$
\caption{Circuit representation of $\mathlarger\prod_{j=0}^{n-1} \CX_{1,n+1-j} R_{Z,n+1}(2t) \mathlarger\prod_{j=0}^{n-1} \CX_{1,n+1-j}\,.$}
\label{fig:circuit_exp_proof}
\end{figure}
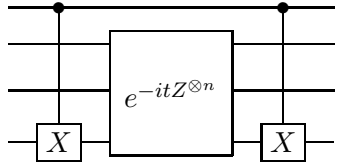
Thus, 
\begin{align*}
    \prod_{j=0}^{n-1} \CX_{1,n+1-j} R_{Z,n+1}(2t) \prod_{j=0}^{n-1} \CX_{1,n+1-j} &= \CX_{1,n+1}(I\otimes  e^{-itZ^{\otimes n}})\CX_{1,n+1} \\
    &= \begin{pmatrix}
        I^{\otimes n}&0\\0&I^{\otimes n-1}\otimes X
    \end{pmatrix}\begin{pmatrix}
        e^{-itZ^{\otimes n}}&0\\0&e^{-itZ^{\otimes n}}
    \end{pmatrix}\begin{pmatrix}
        I^{\otimes n}&0\\0&I^{\otimes n-1}\otimes X
    \end{pmatrix}\\
    &= \begin{pmatrix}
        e^{-itZ^{\otimes n}}&0\\0&(I^{\otimes n-1}\otimes X)e^{-itZ^{\otimes n}}(I^{\otimes n-1}\otimes X)
    \end{pmatrix}\\
    &= \begin{pmatrix}
        e^{-itZ^{\otimes n}}&0\\0&e^{itZ^{\otimes n}}
    \end{pmatrix}\,,
\end{align*}
where the last line comes from the application of Lemma~\ref{lem:induction}. Thus, 
$$ \Exp(Z^{\otimes n+1}, t) = \prod_{j=0}^{n-1} \CX_{1,n+1-j} R_{Z,n+1}(2t) \prod_{j=0}^{n-1} \CX_{1,n+1-j}\,,$$
proving $P_{n+1}$.
\end{proof}

\end{document}